\documentclass[12pt]{article}
\usepackage{bbm, amsmath, amsfonts, amscd, latexsym, amsthm, amssymb, graphicx, euscript, mathrsfs,refcount,hyperref}
\usepackage{mathtools}
\usepackage[nottoc,notlot,notlof]{tocbibind}
\usepackage[dvipsnames]{xcolor}
\usepackage{tikz}
\usetikzlibrary{patterns}

\newtheorem{theor}{\hspace{1cm}{\sc Theorem}}[section]
\newtheorem{utver}[theor]{\hspace{1cm}{\sc Proposition}}
\newtheorem{predpol}[theor]{\hspace{1cm}{\sc Assumption}}
\newtheorem{sledst}[theor]{\hspace{1cm}{\sc Corollary}}
\newtheorem{lemma}[theor]{\hspace{1cm}{\sc Lemma}}
\newtheorem{conj}[theor]{\hspace{1cm}{\sc Conjecture}}
\newtheorem*{utver*}{\hspace{1cm}{\sc Proposition}}
\theoremstyle{definition}
\newtheorem{defin}[theor]{\hspace{1cm}{\sc Definition}}
\newtheorem{exa}[theor]{\hspace{1cm}{\sc Example}}
\newtheorem{rem}[theor]{\hspace{1cm}{\sc Remark }}
\newtheorem{prb}[theor]{\hspace{1cm}{\sc Problem}}

\newcommand{\ind}{\mathop{\rm ind}\nolimits}

\newcommand{\vol}{\mathop{\rm Vol}\nolimits}
\newcommand{\area}{\mathop{\rm Area}\nolimits}
\newcommand{\length}{\mathop{\rm Length}\nolimits}
\newcommand{\conv}{\mathop{\rm conv}\nolimits}

\newcommand{\supp}{\mathop{\rm supp}\nolimits}
\newcommand*\circled[1]{\tikz[baseline=(char.base)]{
\node[shape=circle,draw,inner sep=2pt] (char) {#1};}}

\newcommand{\bigslant}[2]{{\raisebox{.2em}{$#1$}\left/\raisebox{-.2em}{$#2$}\right.}}
\def\R{\mathbb R}
\def\N{\mathbb N}
\def\Z{\mathbb Z}

\def\C{\mathbb C}
\def\CC{({\mathbb C}\setminus 0)}

\begin{document}
\title{The Newton Polytope of the Morse Discriminant of a Univariate Polynomial}
\author{Arina Voorhaar\thanks{{\it National Research University Higher School of Economics}, Moscow, Russia}\thanks{{\it University of Geneva}, Geneva, Switzerland}}
\date{}
\maketitle{}
\begin{abstract}
In this paper we compute the Newton polytope $\mathcal M_A$ of the Morse discriminant in the space of univariate polynomials with the given support set $A.$ Namely, we establish a surjection between the set of all combinatorial types of Morse univariate tropical polynomials and the vertices of $\mathcal M_A.$
\end{abstract}
\tableofcontents
\section{Introduction}

For an arbitrary support set $A\subset\Z^n,$ the Newton polytope of the  $A$--discriminant $D_A$ (i.e. the closure of the set of all non-smooth hypersurfaces given by polynomials with the support set $A$) was described by Gelfand, Kapranov and Zelevinsky in \cite{GKZ}. Namely, they established a surjection from the set of regular triangulations of the polytope $\conv(A)$ with vertices in $A$ (or, equivalently, the set of combinatorial types of smooth tropical hypersurfaces with support $A$) to the vertices of the Newton polytope of $D_A.$ 

In this paper, we obtain a similar result for the Newton polytope of the Morse discriminant in the space $\C^A$ of univariate polynomials with the given support $A\subset\Z.$  In other words, we are interested in the following codimension $1$ strata in $\C^A.$ 

\begin{defin}
The {\it caustic} in the space of polynomials with the given support $A$ is the set of all Laurent polynomials $f\in\C^A$ such that the map $f\colon\CC^n\to\C$ has a degenerate critical point. 
\end{defin}

\begin{defin}
The {\it Maxwell stratum} in the space of polynomials with the given support $A$ is the set of all Laurent polynomials $f\in\C^A$ such that the map $f\colon\CC^n\to\C$ has a pair of coinciding critical values taken at distinct points.
\end{defin}

\begin{defin}
A polynomial $f\in\C^A$ is called {\it Morse,} if it belongs neither to the caustic, nor to the Maxwell stratum.
\end{defin}

\begin{defin}
The {\it Morse discriminant} is the closure of the set of all non-Morse polynomials $f\in\C^A.$ It is given by the polynomial $h_m^2h_c,$ where $h_m$ and $h_c$ are polynomials defining the Maxwell stratum and the caustic, respectively, if these two sets are hypersurfaces. Otherwise we set the corresponding defining polynomial to $1$. 
\end{defin}

\begin{rem}
There are various ways to define the Morse discriminant, and our computational method can be adjusted to suit any of them. See Remark \ref{maxwell_supp} for details.
\end{rem}

The goal of this paper is to describe the Newton polytope $\mathcal M_{A}$ of the Morse discriminant, provided that the set $A$ satisfies the following property. 

\begin{predpol}\label{gen_sing}
For a generic polynomial $f$ in the Morse discriminant, the map $f\colon\CC\to\C$ has exactly one pair of coinciding critical values or exactly one degenerate critical point (of multiplicity 2).
\end{predpol}

\begin{conj}\label{sing_conj}
Any set $A$ with $\length(\conv (A))\geqslant 3,$ which affinely generates the lattice $\Z,$ satisfies Assumption \ref{gen_sing}.
\end{conj}

\begin{rem}
Even though the statement above is a conjecture, there is a wide range of sets $A$ satisfying Assumption \ref{gen_sing}. For instance, any set $A\subset\Z$ such that  ${A=\conv(A)\cap\Z}$ and $\length(\conv (A))\geqslant 3,$ as well as any set $A\subset Z$ containing $4$ consecutive integers, satisfies Assumption \ref{gen_sing}. 
\end{rem}

The problem of describing the Newton polytope and other closely related invariants of the Morse discriminant was studied by various authors. The degree of the Morse discriminant for general degree $d$ univariate polynomials was computed in \cite{LZ}. The tropical fan of the variety of univariate degree $d$ polynomials having two multiple roots was studied in \cite{DHT} and in \cite{E3} (in a more general setting). The maximal cones of the tropical fan, that were computed in these works, under the projection along a line spanned by a constant monomial, define the directions of all the edges of the Morse polytope. However, due to non-trivial intersections of the images of the cones under this projection, the results obtained in \cite{DHT} and in \cite{E3} cannot be directly used to enumerate the edges and vertices of the Morse polytope.

The results obtained in this paper motivate the following definitions of the tropical versions of the Maxwell stratum and caustic in the space of univariate tropical polynomials with the given support $A.$

\begin{defin}\label{trop_Maxwell}
We say that a tropical Laurent polynomial $F(X)$ belongs to the {\it tropical Maxwell stratum} in the space of tropical polynomials with the given support $A,$ if there exists a pair $r_1,r_2$ of tropical roots of $F(X),$ such that $F(r_1)=F(r_2).$ 
\end{defin}

\begin{defin}\label{trop_caustic}
A tropical Laurent polynomial $F(X)$ belongs to the {\it tropical caustic} in the space of tropical polynomials with the given support $A,$ if for some tropical root $r$ of $F(X),$ there are at least two pairs of monomials attaining the same values at $r.$ 
\end{defin}

\begin{defin}\label{trop_Morse_defin}
A tropical Laurent polynomial $F(X)$ is called {\it Morse} if it belongs neither to the Maxwell stratum, nor to the caustic.
\end{defin}

It is natural to try to interpret Definitions \ref{trop_Maxwell} and \ref{trop_caustic} in the language of tropical covers developed by numerous authors (see e.g. \cite{BBM},\cite{CJM}) for a tropical approach to computing Hurwitz numbers. For a detailed exposition of tropical covers and their applications we refer the reader to \cite{B}. 

Let $F(X)$ be a tropical Laurent polynomial in one variable. Consider the projection of the graph of $F(X)$ onto the vertical axis. This projection gives rise to a map  $\rho\colon\Gamma\to\Gamma'$ of abstract tropical curves (graphs), which is a tropical cover. Then the polynomial $F(X)$ belongs to the tropical Maxwell stratum, if and only if, in the notation of \cite{B}, for some internal vertex $v'\in\Gamma',$ the map $\rho$ is not simply ramified over $v'.$ It would be interesting to interpret the definition of the tropical caustic in the same terms, but this would probably require some more refined definitions of an abstract tropical curve and a tropical cover.

All the combinatorial types of tropical Morse polynomials can be encoded by certain combinatorial data described in Subsection \ref{cd}. As the main result of this paper, we establish a surjection between the set of all possible combinatorial types of tropical Morse polynomials with given support $A$ and the vertices of the Newton polytope $\mathcal M_{A}$ of the Morse discriminant in $\C^A.$ We denote by $\mu_{A}$ the support function of the polytope $\mathcal M_A.$ It is given by formula (\ref{main_formula}) in Theorem \ref{support}. Then the abovementioned surjection is defined as follows: first we consider the tuple of coefficients of the given Morse tropical Laurent polynomial $F(X)$ as a covector $\gamma\in(\R^{|A|})^*,$ and then we compute the differential $d\mu_{A}$ at $\gamma.$ Thus we obtain the description of the sought Newton polytope $\mathcal M_{A}$ as the convex hull of the points in the image of the surjection described above.

The paper is organized as follows. In Section \ref{prelim}, we give an overview of the results obtained in \cite{V} that will be used to compute the function $\mu_A$ in Section \ref{the_proof}. The detailed description of the main result together with examples of its usage is given Section \ref{answer}. For the convenience of the reader, Section \ref{answer} is self-contained and can be read separately from the other sections. Finally, Section \ref{the_proof} is devoted to the proof of the main result of this paper and is split into 6 subsections according to the 6 major steps of the proof. \\

{\bf Acknowledgements.} I would like to thank my supervisor Alexander Esterov for exceptional support, guidance and fruitful discussions. The research was supported by the NCCR SwissMAP of the Swiss National Science Foundation.\\

\section{Preliminaries}\label{prelim}
In this Section we recall the definition of a forking-path singularity and state one of the key results that will be used in Section \ref{the_proof}. See \cite{V} for more details and examples.
\subsection{Forking Paths Singularities}\label{intro_fps}
This subsection is devoted to the so-called {\it forking paths singularities}, introduced in the work \cite{E3}. 

Let $i=(i_1,i_2\ldots)$ be a sequence of integers satisfying the following properties: 
\begin{itemize}
	\item the sequence $i$ stabilizes at $1;$
	\item for every $r\in\N,$ the number $i_{r+1}$ divides $i_r.$ 
\end{itemize}
Given such a sequence $i,$ one can construct another sequence $q=(q_1,q_2\ldots)$ as follows: for every $r\in\N,$ set $q_{r}=\dfrac{i_r}{i_{r+1}}.$

This data can be encoded using a certain system of subsets of a finite set $R$ of $i_1$ elements via the so-called {\it $i$-nested boxes construction}. This construction works level by level as follows. Level $0$ consists of one box -- the set $R$ itself. To construct level $1,$ we divide the elements in $R$ into $q_1$ boxes containing $i_2$ elements each. Level $2$ is then the result of dividing the elements of each of the level $1$ boxes into $q_2$ boxes containing $i_3$ elements each. We continue this operation until we end up with $i_1$ boxes, each containing an element of $R$. The latter will happen in a finite number of steps, since the sequence $i$ stabilizes at $1.$ 

Each of the elements of $R$ in the nested-boxes construction has its own {\it address}, i.e. a finite sequence of integers, constructed as follows. The $(k+1)-$th element of the address is the number of the $k-$th level box containing the given element. For any two elements $r_1,r_2$ of $R$ with addresses $(a_1,a_2,\ldots,a_N)$ and $(b_1,b_2,\ldots,b_N)$ one can define the {\it depth} of their relation as the number $\kappa(r_1,r_2)$ equal to the minimal number $K$ such that $a_K\neq b_K.$

\begin{defin}\label{def_fps}
	In the same notation as above, let $i=(i_1,i_2,\ldots)$ be an integer sequence stabilizing at $1$ and such that $i_{r+1}$ divides $i_r$ for every $r.$ With the $i-$nested boxes construction one can associate a plane singularity with $i_1$ distinct regular branches $\varphi_{r_m}\colon(\C,0)\to(\C^2,0)$ indexed by the elements of the set $R$ and such that the intersection number of $\varphi_{r_m}$ and $\varphi_{r_n}$ with $i\neq j$ is equal to $\kappa(r_m,r_n).$ We call this singularity an {\it $i-$forking paths singularity.}
\end{defin}

\begin{utver}\label{chifps}
	The Euler characteristic $\chi(i)$ of the Milnor fiber of an $i-$forking paths singularity can be computed using the following formula:  
	\begin{equation}
	\chi(i)=i_1-i_1\sum\limits_{n=1}^{\infty}(i_n-1).
	\end{equation}
\end{utver}

\subsection{Singular points of a plane projection of a 1--dimensional complete intersection}\label{proj_sing}
This subsection concerns the singular points of a plane projection of a generic $1-$dimensional complete intersection given by a system of equations with the given support set $\tilde{A}\subset \Z^3.$ For more general results, details and proofs, we refer the reader to \cite{V}. 

Choose a coordinate system $(x,y,t)$ in $\CC^3$ and the corresponding basis $\{e_1,e_2,e_3\}$ in the character lattice $L\simeq\Z^3.$ Let $\tilde{A}\subset\Z\langle e_1,e_2,e_3\rangle$ be a finite set of full dimension and $\Delta\subset\R^3$ be its convex hull. By $\ind_v(\tilde{A})$ we denote the index of the sublattice in $\Z$ generated by the image of $\tilde{A}$ under the projection $\rho\colon\Z^{3}\twoheadrightarrow\bigslant{Z^{3}}{\langle e_{2},e_{3}\rangle}.$ 

The main statement of this subsection uses the following two assumptions. 

\begin{predpol}\label{indall}
The set $\tilde{A}$ contains $0\in\Z^3$ and $\ind_v(\tilde{A})=1.$
\end{predpol}

\begin{predpol}\label{primitive0}
	For every facet of $\Delta$ the image of its primitive normal covector under the projection forgetting the first coordinate is also primitive.
\end{predpol}

Let $Q=\rho(\Delta)$ be the image of the polytope $\Delta$ under the projection $\rho\colon\R^{3}\twoheadrightarrow\bigslant{\R^{3}}{\langle e_2,e_3\rangle}.$ 

\begin{defin}
	We call a face $\tilde{\Gamma}\subset\Delta$ {\it horizontal}, if its projection is contained in the boundary of $Q$. 
\end{defin}

Let $\Gamma\subset\Delta$ be a non-horizontal facet contained in a hyperplane given by a linear equation of the form $h(e_1,e_2,e_3)=\lambda.$ The function $h$ is unique up to a scalar multiple, therefore, one can assume that the coefficients of $h$ are coprime integers and that for any point $\alpha\in A\setminus\Gamma,$ we have $h(\alpha)<\lambda.$ 

We now construct a sequence of integers $i^{\Gamma}=(i_1^{\Gamma},i_2^{\Gamma},\ldots)$ as follows. 

Set $B_1^{\Gamma}=\tilde A\cap\Gamma.$ For every $r>1,$ we define  $$B_r^{\Gamma}=B_{r-1}^{\Gamma}\cup(\tilde A\cap\{h(e_1,e_2,e_3)=\lambda-(r-1)\}).$$ 
Finally, for every $r\geqslant 1,$ we set $$i_r^{\Gamma}=\ind_{v}(B_r^{\Gamma}).$$

It is clear that for every $r,$ the element $i_r^{\Gamma}$ divides $i_{r-1}^{\Gamma}.$ Moreover, since for the set $\tilde A$ we have $\ind_v(A)=1,$ any such sequence stabilizes to $1$.

Consider a complete intersection $\mathcal C=\{f_1=f_2=0\}\in \CC^3$ with $\supp(f_1)=\supp(f_2)=\tilde{A},$ and the projection $\pi\colon\CC^3\to\CC^2$ forgetting the first coordinate. For generic $f_1,f_2\in\C^{\tilde A},$ the closure of the image $\pi\tilde{\mathcal C}\subset\CC^2$ is an algebraic curve $D$ with finitely many isolated singular points. Its Newton polygon is equal to the fiber polytope $P=\int_{\pi}\Delta$ (see \cite{EK}). 

The curve $D$ is not in general Newton non-degenerate with respect to its Newton polygon. In other words, the closure of $D$ in the toric surface $X_P$ might intersect its $1-$dimensional orbits not transversally. One of the clue results of the paper \cite{V} is the description of the singularities of the curve $D$ at infinity for sufficiently generic tuples $f_1,f_2\in\C^{\tilde{A}}.$

\begin{lemma} \label{fpsinfinity}
Under Assumptions \ref{indall} and \ref{primitive0}, for a generic tuple of polynomials $f_1,f_2\in\C^{\tilde A},$ all the singularities of the curve $D$ at infinity are $i^{\Gamma}-$forking paths singularities for non-horizontal facets $\Gamma\subset\Delta.$ 
\end{lemma}

There seems to be no elegant way to formulate the necessary genericity conditions required in Lemma \ref{fpsinfinity} for general $\tilde{A}.$ However, in this paper, Lemma \ref{fpsinfinity} will be applied to some particular case that we will cover in Example \ref{fpsgen} below. For more details about the general case, see Lemma 4.13 in \cite{V}.  

\begin{exa}\label{fpsgen}
Consider a pair of polynomials $f(x,y,t)=1+x^l+y^k+\sum\limits_{i>0} c_it^i x^{m_i}y^{n_i}$ and $g(x,y,t)=a+bx^l+\sum\limits_{i>0} \alpha_ic_it^i x^{m_i}y^{n_i},$ where $l>1,~k>0,$ the coefficients $a,b,\alpha_i$ are some fixed nonzero and pairwise distinct numbers. Some of the coefficients $c_i$ are $0$, the other numbers $c_i$ are chosen generically. Suppose that the sets $\tilde A=\supp(f),~\Delta=\conv(\tilde A)$ are of maximal dimension and satisfy Assumption \ref{indall}. Let $(x_1,y_0)$ and $(x_2,y_0)$ be a pair of roots for the truncated system $\{1+x^l+y^k=a+bx^l=0\}.$ Consider the branches of the closure of $\mathcal C=\{f(x,y,t)=g(x,y,t)=0\}$ passing through the points $p_1=(x_1,y_0,0)$ and $p_2=(x_2,y_0,0).$ These branches intersect the plane $\{t=0\}$ transversally. We need to compute the order of contact at $p=(y_0,0)$ between their images under the projection $\pi$ forgetting the first coordinate. Let $K(p_1,p_2)$ be the smallest index $i$ such that $x_1^{m_i}y_0^{n_i}\neq x_2^{m_i}y_0^{n_i}.$ By Proposition 4.14 in \cite{V}, the sought order of contact is greater or equal to $K(p_1,p_2).$ 

First, consider the case $K(p_1,p_2)=1.$ Let us show that the corresponding order of contact is equal to $1,$ that is, the corresponding branches of the curve $D$ intersect transversally at $p.$ To do this, we will compare the projections of the tangent lines to $\mathcal C$ at $p_1$ and $p_2.$ If they do not coincide, then the order of contact is indeed $1.$ The coefficients of the equations defining the projections of the corresponding tangent lines are equal to the last $2$ minors of the Jacobian matrix evaluated at the points $p_1$ and $p_2$ respectively. In our case, we need to check that $$\begin{vmatrix}M_2(p_1) & M_3(p_1)\\M_2(p_2) & M_3(p_2)\end{vmatrix}\neq 0,$$ where $M_2$ and $M_3$ are the last two minors of the matrix
$$\mathcal J=\begin{pmatrix}l\cdot x^{l-1} & k\cdot y^{k-1} & x^{m_1}y^{n_1}\\b\cdot l\cdot x^{l-1} & 0 & \alpha_1x^{m_1}y^{n_1} \end{pmatrix}.$$
If this determinant is equal to $0,$ then for some $\lambda,$ we have 
$$\begin{cases}
b\cdot l\cdot k \cdot x_1^{l-1}y_0^{k-1}=\lambda\cdot b\cdot l\cdot k \cdot x_2^{l-1}y_0^{k-1}\\
x_1^{m_1}y_0^{n_1}(\alpha_1l\cdot x_1^{l-1}-b\cdot l\cdot x_1^{l-1})=\lambda\cdot x_2^{m_1}y_0^{n_1}(\alpha_1\cdot l\cdot x_2^{l-1}-b\cdot l\cdot x_2^{l-1}).
\end{cases}$$
The first equality implies that $\lambda=\Big(\dfrac{x_1}{x_2}\Big)^{l-1}.$ Substituting this expression into the second equality, we obtain that $$x_1^{m_1+l-1}y_0^{n_1}\cdot l(\alpha_1-b)=\Big(\dfrac{x_1}{x_2}\Big)^{l-1}x_2^{m_1+l-1}y_0^{n_1}\cdot l(\alpha_1-b).$$ The latter is equivalent to $x_1^{m_1}=x_2^{m_1}.$ On the other hand, $K(p_1,p_2)=1.$ Thus we have $x_1^{m_1}\neq x_2^{m_1},$ which leads us to a contradiction. 

The statement for $K=K(p_1,p_2)>1$ follows from the computation above by the following genericity argument (non-trivial in the sense that it does not produce an explicit genericity condition on the coefficients $c_i$ for $K>1$). Indeed, if there are no other monomials of the form $t^ix^{m_i}y^{n_i}, i\neq K,$ then the change of variables $\check{t}=t^K$ reduces this case to the one considered above. Moreover, the monomials of the form $t^ix^{m_i}y^{n_i},~i> K,$ do not change the sought order of contact. Finally, we need to deal with the monomials $t^ix^{m_i}y^{n_i},~i<K.$ Let us consider the space parameterized by all possible choices of coefficients $c_i,i<K.$ The sought order of contact can be viewed as an integer-valued function constructively depending on the choice of the coefficients $c_i$ (i.e. representable as a finite linear combination of the indicator functions of algebraic sets). We already know its lower bound -- it is equal to $K.$ Moreover, if the coefficients $c_i,~i<K,$ are all equal to $0,$ then the desired order of contact is exactly $K.$ By upper semi-continuity of the intersection index the latter means that for generic choice of the coefficients $c_i,~i<K,$ the order of contact does not exceed $K,$ and thus, is equal to $K.$

Finally, by definition \ref{def_fps}, we obtain that the singularity of the curve $D$ at the point $p$ is indeed the forking--path singularity mentioned in Lemma \ref{fpsinfinity}.
\end{exa}

\section{Statement of the Main Result}\label{answer}
In this Section we give the formula for the support function $\mu_A$ of the Newton polytope $\mathcal M_A$ of the Morse discriminant in the space of univariate polynomials with given support $A\subset\Z$ and provide some examples of its usage. Let $A=\{a_0,\ldots,a_{|A|-1}\}\subset \Z\setminus\{0\}$ be a finite set affinely generating the lattice $\Z,$ satisfying Assumption \ref{gen_sing} and the condition  $\length(\conv (A))\geqslant 3.$

\subsection{Combinatorial types of tropical polynomials}\label{cd}
Any covector with non-negative entries in $\gamma\in(\R^{|A|})^*$ can be viewed as a function ${\gamma\colon A\to\R_{\geqslant 0}}.$ 
Denote by $N_{\gamma}$ the polygon $$N_{\gamma}=\conv\Big(\{(a_p,\gamma(a_p))\mid a_p\in A\}\cup\{(a_0,0),(a_{|A|-1},0)\}\Big).$$ 
The projection of its upper edges onto the first coordinate yields a subdivision of the interval $\conv(A)$ with vertices in $A.$ One can encode this subdivision by the subset ${W_{\gamma}=\{w_0,\ldots,w_k\}\subset A}$
containing all its vertices. We have $w_0=a_0$ and $w_k=a_{|A|-1}.$  

We now formulate the genericity assumptions on the covector $\gamma$ that we will need for our constructions. 
\begin{predpol}\label{slopes}
For any two distinct pairs $p<q$ and $r<s$ of elements in $A,$ we have 
\begin{equation*}
\dfrac{\gamma(p)-\gamma(q)}{q-p}\neq\dfrac{\gamma(s)-\gamma(r)}{r-s}.
\end{equation*}	
\end{predpol}

\begin{rem}\label{slopes_r1}
One can reformulate Assumption \ref{slopes} as follows. For any two distinct pairs $p<q$ and $r<s$ of elements in $A,$ the slopes of the intervals $[(p,\gamma(p)),(q,\gamma(q))]$ and $[(r,\gamma(r)),(s,\gamma(s))]$ do not coincide. 
\end{rem}

\begin{predpol}\label{roots}
For any $0\leqslant j<m\leqslant k-1,$ we have that 
\begin{equation*}
\dfrac{w_{j+1}\gamma(w_j)-w_j\gamma(w_{j+1})}{w_{j+1}-w_j}\neq\dfrac{w_{m+1}\gamma(w_m)-w_m\gamma(w_{m+1})}{w_{m+1}-w_m}.
\end{equation*}	
\end{predpol}

\begin{rem}
Denote by $\ell_j$ the line passing through the points $(w_j,\gamma(w_j))$ and $(w_{j+1},\gamma(w_{j+1})).$ Alternatively, Assumption \ref{roots} can be formulated as follows. For any two indices $0\leqslant j<m\leqslant k-1,$ the lines $l_j,l_m$ and $\{e_1=0\}$ do not meet at the same point. Also note that this condition is satisfied automatically for sets $A$ such that $0\notin\conv(A).$
\end{rem}

On the other hand, the covector $\gamma$ can be viewed as the tropical polynomial  $$\varphi_{\gamma}(X)=\bigoplus\limits_{a\in A_d}\gamma(a)\odot X^{\odot a}.$$

Note that each of the objects listed below depends on the covector $\gamma,$ but to simplify the notation 
we will omit the subscript $\gamma.$ For example, we shall write $W$ instead of $W_{\gamma}$ and $N$ instead of $N_{\gamma}.$ 
\begin{defin}\label{def_comb_data}
Using the polynomial $\varphi_{\gamma}$ we will extract the following {\it combinatorial data.} 
\begin{itemize}
	\item[--] the subset $W=\{w_0\le w_1,\ldots\le w_k\}\subset A$ with $w_0=a_0$ and $w_k=a_{|A|-1},$ that can alternatively be defined as the subset consisting of all the exponents $p\in A$ such that for some $X_0$ we have $\varphi_{\gamma}(X_0)=\gamma(p)\odot X_0^{\odot p}.$
	\item[--] For every $0\leqslant j<k,$ let us evaluate the monomial $\gamma(w_j)\odot X^{\odot w_j}$ at the corresponding root $r_j=\dfrac{\gamma(w_j)-\gamma(w_{j+1})}{w_{j+1}-w_j}$ of the tropical polynomial $\varphi_{\gamma}$ and place all the numbers $j\in\{0,1,\ldots,k-1\}$ in increasing order with respect to the values attained. This yields a sequence $Z=(z_1,z_2,\ldots,z_k).$ 
	\item[--] a sequence $M^j$ for every root $r_j$ of the tropical polynomial $\varphi_{\gamma},$ constructed as follows. The elements of $M^j$ are all the points in $p\in A\setminus\{w_j,w_{j+1}\}$ placed in $M^j$ in decreasing order with respect to the values of the monomials $\gamma(p)\odot X^{\odot p}$ attained at $r_j.$
\end{itemize}
\end{defin}
\begin{exa}\label{comb_data}
Take $A=\{-3,-1,1,2,4\}$ and $\gamma=(3,5,2,5,1).$ Thus we obtain 
\begin{multline*}
\varphi_{\gamma}=3\odot X^{\odot -3}\oplus 5\odot X^{\odot -1}\oplus 2\odot X^{\odot 1}\oplus 5\odot X^{\odot 2}\oplus 1\odot X^{\odot 4}=\\=\max\{-3X+3,~-X+5,~X+2,~2X+5,~4X+1\}.
\end{multline*}
Its graph is shown in Figure 1 below. We have $W=\{w_0,w_1,w_2,w_3\}=\{-3,-1,2,4\}.$ Moreover, since  $\varphi_{\gamma}(r_1)<\varphi_{\gamma}(r_0)<\varphi_{\gamma}(r_2)$, we have $Z=(1,0,2).$ Finally, we obtain $$M^0=(2,1,4),~M^1=(-3,1,4),~M^2=(1,-1,-3).$$
\begin{center}
	\begin{tikzpicture}[scale=0.5]
	\draw[->,thick] (0,0)--(0,18);
	\draw[->,thick] (0,0)--(20,0);
	\draw[thick] (2,15)--(15,2);
	\draw[thick, dashed] (0,17)--(2,15);
	\draw[thick, dashed] (15,2)--(17,0);
	\draw[thick] (3,2)--(11,18);
	\draw[thick, dashed] (2,0)--(3,2);
	\draw[thick] (6,2)--(10,18);
	\draw[thick, dashed] (6,2)--(5.5,0);
	\draw[thick] (3.67,18)--(9.33,1);
	\draw[thick,dashed] (9.33,1)--(9.67,0);
	\draw[thick] (2,2)--(17,17);
	\draw[thick,dashed] (0,0)--(2,2);
	\draw[thick, dashed, blue] (6,2)--(6,0);
	\draw[thick, blue] (6,11)--(6,2);
	\draw[ultra thick, blue] (3.67,18)--(6,11);
	\draw[ultra thick, blue] (7,10)--(6,11);
	\draw[thick, blue] (7,10)--(7,2);
	\draw[thick, dashed, blue] (7,2)--(7,0);
	\draw[ultra thick, blue] (7,10)--(9,14);
	\draw[thick, dashed, red] (9,14)--(0,14);
	\draw[thick, dashed, red] (7,10)--(0,10);
	\draw[thick, dashed, red] (6,11)--(0,11);
	\draw[thick, blue] (9,14)--(9,2);
	\draw[thick, dashed, blue] (9,2)--(9,0);
	\draw[ultra thick, blue] (9,14)--(10,18);
	\node[below left,scale=0.6] at (0,0) {$(-\infty,-\infty)$};
	\node[below right,scale=0.7] at (14,14) {$X+2$};
	\node[right,scale=0.7] at (10,16) {$2X+5$};
	\node[above,scale=0.7] at (2,16) {$-X+5$};
	\node[right,scale=0.7] at (4,17) {$-3X+3$};
	\node[left,scale=0.7] at (9.5,15.5) {$4X+1$};
	\node[below,scale=0.7] at (6,0) {$r_0$};
	\node[below,scale=0.7] at (7,0) {$r_1$};
	\node[below,scale=0.7] at (9,0) {$r_2$};
	\node[left,scale=0.7,red] at (6,8) {$2$};
	\node[left,scale=0.7,red] at (6,6) {$1$};
	\node[left,scale=0.7,red] at (6,2) {$4$};
	\node[above right,scale=0.7,red] at (6.7,8.2) {$-3$};
	\node[left,scale=0.7,red] at (7,7) {$1$};
	\node[left,scale=0.7,red] at (7,6) {$4$};
	\node[right,scale=0.7,red] at (9,9) {$1$};
	\node[right,scale=0.7,red] at (9,8) {$-1$};
	\node[right,scale=0.7,red] at (9,2) {$-3$};
	\node[left, scale=0.7] at (0,10) {$\varphi_{\gamma}(r_1)$};
	\node[left, scale=0.7] at (0,11) {$\varphi_{\gamma}(r_0)$};
	\node[left, scale=0.7] at (0,14) {$\varphi_{\gamma}(r_2)$};
	\filldraw[color=red, fill=red!70, thick](6,8) circle (0.1);
	\filldraw[color=red, fill=red!70, thick](6,6) circle (0.1);
	\filldraw[color=red, fill=red!70, thick](6,2) circle (0.1);
	\filldraw[color=red, fill=red!70, thick](7,6) circle (0.1);
	\filldraw[color=red, fill=red!70, thick](7,8) circle (0.1);
	\filldraw[color=red, fill=red!70, thick](7,7) circle (0.1);
	\filldraw[color=red, fill=red!70, thick](9,9) circle (0.1);
	\filldraw[color=red, fill=red!70, thick](9,8) circle (0.1);
	\filldraw[color=red, fill=red!70, thick](9,2) circle (0.1);
	\node[below] at (8,-1.5) {Figure 1. The graph of the tropical polynomial $\varphi_{\gamma}$.};
	\end{tikzpicture}
\end{center}
\end{exa}

\begin{rem}
The order if the elements in the sequences $M^j$ and $Z$ is well-defined due to Assumptions \ref{slopes} and \ref{roots} on the covector $\gamma.$ For any two lines of the form $pX+\gamma(p)$ and $qX+\gamma(q),$ the $X-$coordinate of their intersection is equal to $\dfrac{\gamma(p)-\gamma(q)}{q-p}.$ If it coincides with a root $r_j$ of the tropical polynomial $\varphi_{\gamma},$ then we have the equality $$\dfrac{\gamma(p)-\gamma(q)}{q-p}=\dfrac{\gamma(w_j)-\gamma(w_{j+1})}{w_{j+1}-w_j}.$$ The latter contradicts Assumption \ref{slopes}. 

Moreover, if $\varphi_{\gamma}(r_j)=\varphi_{\gamma}(r_m)$ for some $j<m,$ then we have the following equality: $w_jr_j+\gamma(w_j)=w_mr_m+\gamma(w_m).$ Substituting the expressions $r_j=\dfrac{\gamma(w_j)-\gamma(w_{j+1})}{w_{j+1}-w_j},$ we obtain $$\dfrac{w_{j+1}\gamma(w_j)-w_j\gamma(w_{j+1})}{w_{j+1}-w_j}=\dfrac{w_{m+1}\gamma(w_m)-w_m\gamma(w_{m+1})}{w_{m+1}-w_m},$$
which contradicts Assumption \ref{roots}.
\end{rem}

\begin{rem}\label{slopes_r2}
In Remark \ref{slopes_r1} we noted that Assumption \ref{slopes} can be reformulated in more geometric terms, since the number $\dfrac{\gamma(p)-\gamma(q)}{q-p}$ is equal to the slope of the interval $[(p,\gamma(p)), (q,\gamma(q))].$ Thus, to determine the order of elements in the sequence $M^j,$ we need to compare the slopes of all possible intervals $[(p,\gamma(p)), (q,\gamma(q))],~p,q\in A\setminus\{w_j,w_{j+1}\},$ with the slope of $[(w_j,\gamma(w_j)),(w_{j+1},\gamma(w_{j+1}))].$
\end{rem}

\begin{rem}\label{trop_gen_cond}
The computation above implies that tropical polynomial $\varphi_{\gamma}(X)$ is Morse in the sense of Definition \ref{trop_Morse_defin} if and only if the covector $\gamma$ satisfies Assumptions \ref{slopes} and \ref{roots}. 
\end{rem}

\begin{utver}\label{cones}
	All the covectors $\gamma\in(\R^{|A|})^*$ giving the same combinatorial data generate a full-dimensional convex cone. 
\end{utver}

\begin{rem}
Due to Remark \ref{trop_gen_cond}, the cones from Proposition \ref{cones} can be viewed as combinatorial types of tropical Morse polynomials. Namely, each of the cones can be identified with the tuple of combinatorial data (namely, the subset $W$ and the sequences $M^j$ and $Z$) given by any covector lying inside this cone.
\end{rem}

\subsection{The support function of the polytope \texorpdfstring{$\mathcal M_A$}{}}

Before we give the formula for the support function $\mu_A$ of the polytope $\mathcal M_A$ we need to introduce some more notation. 
For every sequence $M^j=(m_1^j, m_2^j ,\ldots)$ we construct another sequence $b^j=(b_0^j,b_1^j,\ldots)$ defined as follows: $b_0^j=\gcd(w_j,w_{j+1})$ 
and $b_{l+1}^j=\gcd(b_l^j,m_{l+1}^j).$ Note that since $A$ affinely generates the lattice $\Z,$ the sequences $b^j$ stabilize at $1.$ 
We denote by $C^j_{\gamma}$ the following sum:
\begin{equation}\label{FPS_contrib}
C^j_{\gamma}=\sum\limits_{l\geqslant 1}\Big((w_{j+1}-w_j)\gamma(m_l^j))+(m_l^j-w_{j+1})\gamma(w_j)+(w_j-m_l^j)\gamma(w_{j+1})\Big)(b_{l-1}^j-b_l^j).
\end{equation}
\begin{rem}\label{gcd_stab}
It follows immediately from the definition that $C^j_{\gamma}$ is equal to $0$ whenever $\gcd(w_j,w_{j+1})=1.$ 
\end{rem}
\begin{exa}\label{seq_exa}
Consider the support set $A=\{-3,-1,1,2,4\}$ and the covector $\gamma=(3,5,2,5,1)$ from Example \ref{comb_data}. Recall that we have $W=\{w_0,w_1,w_2,w_3\}=\{-3,-1,2,4\}$ and  $M^0=(2,1,4),~M^1=(-3,1,4),~M^2=(1,-1,-3).$ Thus, according to the definition above, we obtain $b^0=(1,1,1),~b^1=(1,1,1)$ and $b^2=(2,1,1).$ The sums $C_{\gamma}^0$ and $C_{\gamma}^0$ are equal to $0$ according to Remark \ref{gcd_stab}. The sum $C^2_{\gamma}$ is non-zero. Indeed, we have: 
$$C^2_j=\Big((w_3-w_2)\gamma(1)+(1-w_3)\gamma(w_2)+(w_2-1)\gamma(w_3)\Big)(2-1)=2\gamma(1)-3\gamma(2)+\gamma(4).$$
\end{exa}

We are now ready to give the formula for the support function $\mu_A$ of the polytope $\mathcal M_A$ which will allow us to find the vertices (up to a shift) of the polytope $\mathcal M_A.$                                                                                                                                                                                                                                                                                                                                                                                                                                                                                                                                                                                                                                                                                                                                                                                                                                                                                                                                                                                                                                                                                                                                                                                                                                                                                                                                                                                                                                                                                                                                                                                                                                                                                                                                                                                                                                                                                                                                                                                                                                                                                                                                                                                                                                                                                                                                                                                                                                                                                                                                                                                                                                                                                                                                                                                                                                                                                                                                                                                                                                                                                                                                                                                                                                                                                                                                                                                                                                                                                                                                                                                                                                                                                                                                                                                                                                                                                                                                                                                                                                                                                                                                                                                                                                                                                                                                                                                                                                                                                                                                                                                                                                                                                                                                                                                                                                                                                                                                                                                                                                                                           
\begin{theor}\label{support}
	In the same notation as above, the value of the support function $\mu_{A}$ of the Morse polytope $\mathcal M_A$ (up to a shift depending on integer constants $c_1$ and $c_2$) at the covector $\gamma$ can be computed via the following formula:
	\begin{multline}\label{main_formula}
	\mu_{A}(\gamma)=\sum\limits_{j=1}^{k} S_{z_j}\Big(d_{z_j}-3+\sum\limits_{l=1}^{j-1}2d_{z_l}\Big)+\sum\limits_{j=0}^{j=k-1}C^j_{\gamma}\\+
	\big((|w_0|-w_0)(w_k-w_0)+c_1\big)\gamma(w_0)+\big((w_k+|w_k|)(w_k-w_0)+c_2\big)\gamma(w_k),
	\end{multline}
	where $d_i=w_{i+1}-w_i$ and $S_i=w_{i+1}\gamma(w_i)-w_i\gamma(w_{i+1}).$
\end{theor}
\begin{rem}
The function $\mu_{A}$ is piecewise linear and convex, and it is linear on each of the cones in $(\R^{|A|})^*$ generated by the covectors $\gamma\in(\R^{|A|})^*$ giving the same combinatorial data. Associating each cone with the set of coefficients of the function $\mu_{A}$ on it defines a surjection onto the set of vertices of the polytope $\mathcal M_A.$ 
\end{rem}

\begin{rem}
The formula in \ref{support} depends on the constants $c_1$ and $c_2$. Since they do not change the structure of the desired polytope, they can be chosen arbitrarily. For example, one can choose $c_1,c_2$ in such a way that the sought polytope is contained in the positive octant and lies as close a possible to the origin in $\R^{|A|}.$ We will discuss the choice of shifts in subsequent sections. 
 \end{rem}

\begin{exa}\label{vert1}
Let us compute the support function $\mu_A$ for the support set $A=\{-3,-1,1,2,4\}$ and the combinatorial data from Examples \ref{comb_data} and \ref{seq_exa}. Recall that we have $W=\{w_0,w_1,w_2,w_3\}=\{-3,-1,2,4\}$ and $Z=(1,0,2).$ Moreover, the summands $C^j$ were already computed in \ref{seq_exa}: namely, $C_{\gamma}^0=C_{\gamma}^1=0$ and $C_{\gamma}^2=2\gamma(1)-3\gamma(2)+\gamma(4).$
Using formula (\ref{main_formula}) from Theorem \ref{support}, we obtain: 
\begin{multline*}
\mu_A(\gamma)=(w_2\gamma(w_1)-w_1\gamma(w_2))(w_2-w_1-3)+(w_1\gamma(w_0)-w_0\gamma(w_1))(2w_2-2w_1+w_1-w_0-3)+\\
(w_3\gamma(w_2)-w_2\gamma(w_3))(2w_2-2w_1+2w_1-2w_0+w_3-w_2-3)+2\gamma(1)-3\gamma(2)+\gamma(4)+\\+
(-2w_0(w_3-w_0)+c_1)\gamma(w_0)+(2w_3(w_3-w_0)+c_2)\gamma(w_3)=\\=\big((w_1(2w_2-w_1-w_0-3)-2w_0(w_3-w_0)+c_1\big)\gamma(w_0)+\\+\big(w_2(w_2-w_1-3)-w_0(2w_2-w_1-w_0-3)\big)\gamma(w_1)+\\+\big(w_3(w_3+w_2-2w_0-3)-w_1(w_2-w_1-3)\big)\gamma(w_2)+\\+\big(2w_3(w_3-w_0)-w_2(w_3+w_2-2w_0-3)+c_2\big)\gamma(w_3)+2\gamma(1)-3\gamma(2)+\gamma(4)=\\=(37+c_1)\gamma(-3)+15\gamma(-1)+2\gamma(1)+33\gamma(2)+(39+c_2)\gamma(4).
\end{multline*}
So, we obtained one of the vertices of the sought polytope $\mathcal M_A.$ Up to a shift of the form  $(c_1,0,0,0,c_2)$ which can be chosen arbitrarily and remains the same for all other vertices, the corresponding vertex is $(37,15,2,33,39).$
\end{exa}

\begin{exa}\label{vert2}
Consider the same set $A=\{-3,-1,1,2,4\}$ as in Example \ref{vert1}. Let us compute one more vertex of the polytope $\mathcal M_A.$ Namely, let us compute $\mu_A$ on the covectors $\gamma$ such that the corresponding subdivision is $W=\{w_0,w_1\}=\{-3,4\}.$ Then, since $\gcd(-3,4)=1,$ the summand $C_{\gamma}^0=0,$ and using (\ref{main_formula}) we obtain: 
\begin{multline*}
\mu_A(\gamma)=(w_1\gamma(w_0)-w_0\gamma(w_1))(w_1-w_0-3)+c_1\gamma(w_0)+c_2\gamma(w_1)-\\-2w_0(w_1-w_0)\gamma(w_0)+2w_1(w_1-w_0)\gamma(w_1)=(58+c_1)\gamma(-3)+(68+c_2)\gamma(4).
\end{multline*}
So, up to the same shift $(c_1,0,0,0,c_2)$ as in Example \ref{vert1}, we obtain another vertex $(58,0,0,0,68).$
\end{exa}

\begin{rem}\label{homog}
Since rescaling the variables or multiplying all the coefficients of a polynomial by the same non-zero number does not affect whether or not it is Morse, the polytope $\mathcal M_A$ should lie in the intersection of two affine hyperplanes in $\R^{|A|}:$ namely, the hyperplanes $\{e_0+\ldots+e_{|A|-1}=d_1\}$ and  $\{a_0\cdot e_0+\ldots+a_{|A|-1}\cdot e_{|A|-1}=d_2\}$ for some $d_1,d_2.$ 
\end{rem}

\begin{exa}\label{homog_exa}
Let us look at the vertices found in Examples \ref{vert1} and \ref{vert2}. We have indeed that $$37+15+2+33+39=126=58+68,$$ and $$(-3)\cdot 37+(-1)\cdot 15+1\cdot 2+2\cdot 33+4\cdot 39=98=(-3)\cdot 58+4\cdot 68.$$
\end{exa}

\subsection{Special case: polynomials}
In this subsection we will consider an important special case of our main problem, namely, we will compute the support function of the polytope $\mathcal M_A$ for an arbitrary support set $A$ such that $0\notin\conv(A).$ In this case, for any covector $\gamma$ the corresponding tropical polynomial $\varphi_{\gamma}$ is monotone, therefore, the sequence $Z$ is fully determined by the subdivision $W.$ 

\begin{exa}\label{trop_poly}
Take $A=\{1,2,3,4\},~\gamma=(1,4,3,3).$ Then, the corresponding tropical polynomial $\varphi_{\gamma}(X)$ is equal to $$1\odot X^{\odot 1}\oplus 4\odot X^{\odot 2}\oplus 3\odot X^{\odot 3}\oplus 3\odot X^{\odot 4}=\max\{X+1,~2X+4,~3X+3,~4X+3\}.$$ Its graph is shown in Figure 2 below. We have $W=\{w_0,w_1,w_2\}=\{1,2,4\},~M^0=(3,4)$ and $M^1=(3,1).$ The sequence $Z$ is completely determined by $W,$ and in this case we have $Z=(0,1).$
\end{exa}
We will now give the formula for the support function $\mu_A$ for a set $A$ such that $0\notin \conv(A).$ Without loss of generality, we can assume $A\subset \Z_{>0},$ since the same formula works for $A\subset \Z_{<0}$ up to the changes of the form $w_j\mapsto -w_{k-j}$ and $\gamma(p)\mapsto\gamma(-p).$
The following result is a straightforward corollary of Theorem \ref{support}. 

\begin{sledst}\label{support_poly}
	In the same notation as above, for any set $A\subset\Z_{>0}$ generating the lattice $\Z,$ the value of the support function $\mu_{A}$ of the Morse polytope $\mathcal M_A$ (up to a shift depending on integer constants $c_1$ and $c_2$) at the covector $\gamma$ can be computed via the following formula:
	\begin{multline}
	\mu_{A}(\gamma)=\big(w_1(w_1-w_0-3)+c_1\big)\gamma(w_0)+\sum\limits_{j=1}^{k-1}(w_{j+1}-w_{j-1})(w_{j-1}+w_j+w_{j+1}-2w_0-3)\gamma(w_j)+\\+\big((w_k-w_{k-1})(2w_k+w_{k-1}-2w_0-3)+3w_k+c_2\big)\gamma(w_k)+\sum\limits_{j=0}^{j=k-1}C^j_{\gamma}.
	\end{multline}
\end{sledst}
\begin{center}
\begin{tikzpicture}[scale=0.5]
\draw[->,thick] (0,0)--(0,20);
\draw[->,thick] (0,0)--(15,0);
\draw[thick] (2,8)--(14,11);
\draw[thick, dashed] (2,7)--(0,5.5);
\draw[thick, dashed] (3.75,0.5)--(4,1);
\draw[thick, dashed] (2,3)--(0,1);
\draw[thick] (2,7)--(14,16);
\draw[thick] (2,3)--(14,15);
\draw[thick] (4,1)--(13,19);
\draw[ultra thick, dashed, blue] (2,8)--(0,7.5);
\draw[ultra thick, blue] (2,8)--(4,8.5);
\draw[thick, blue,dashed] (4,0.6)--(4,0);
\draw[thick, blue] (10,13)--(10,1);
\draw[thick,blue] (4,8.5)--(4,0.6);
\draw[thick, blue, dashed] (10,1)--(10,0);
\draw[ultra thick, blue] (4,8.5)--(10,13);
\draw[ultra thick, blue] (10,13)--(13,19);
\draw[thick, dashed, red] (10,13)--(0,13);
\draw[thick, dashed, red] (4,8.5)--(0,8.5);
\node[left, scale=0.8] at (0,8.5) {$\varphi_{\gamma}(r_0)$};
\node[left, scale=0.8] at (0,13) {$\varphi_{\gamma}(r_1)$};
\node[below left,scale=0.6] at (0,0) {$(-\infty,-\infty)$};
\node[above left,scale=0.7] at (14,9.5) {$X+1$};
\node[above left,scale=0.7] at (14,12) {$3X+3$};
\node[above left,scale=0.7] at (14,15.7) {$2X+4$};
\node[left,scale=0.7] at (11,15) {$4X+3$};
\node[below] at (4,0) {$r_0$};
\node[below] at (10,0) {$r_1$};
\node[left,scale=0.7,red] at (4,5) {$3$};
\node[left,scale=0.7,red] at (4,1) {$4$};
\node[left,scale=0.7,red] at (10,11) {$3$};
\node[below left,scale=0.7,red] at (10,9.8) {$1$};
\filldraw[color=red, fill=red!70, thick](4,1) circle (0.1);
\filldraw[color=red, fill=red!70, thick](4,5) circle (0.1);
\filldraw[color=red, fill=red!70, thick](10,11) circle (0.1);
\filldraw[color=red, fill=red!70, thick](10,10) circle (0.1);
\node[below] at (8,-1.5) {Figure 2. The graph of the tropical polynomial $\varphi_{\gamma}$.};
\end{tikzpicture}
\end{center}
             
\begin{sledst}\label{sp_case}
For a set $A$ of the form $A=[1,n]\cap\Z$ the formula after choosing $c_1=4$ and $c_2=6-6n$ in Theorem \ref{support} is of the following form:
 \begin{multline}
 \mu_{A}(\gamma)=(w_1-2)^2\gamma(1)+\sum\limits_{j=1}^{k-1}(w_{j+1}-w_{j-1})(w_{j-1}+w_j+w_{j+1}-5)\gamma(w_j)+\\+(n-1-w_{k-1})(2n+w_{k-1}-6)\gamma(n)+\sum\limits_{j=0}^{j=k-1}C^j_{\gamma}.
 \end{multline}
\end{sledst}

\begin{exa}\label{seq_poly}
	Let us compute the sums $C^j_{\gamma}$ for $A=\{1,2,3,4\}$ and a covector $\gamma$ giving the same combinatorial data as in Example \ref{trop_poly}. We have $W=\{1,2,4\},~M^0=(3,4)$ and $M^1=(3,1).$ Therefore we obtain $b^0=(1,1,1)$ and $b^1=(2,1,1),$ which yields $C^0_{\gamma}=0$ and $$C^1_{\gamma}=(-\gamma(2)+2\gamma(3)-\gamma(4))(2-1)+(2\gamma(1)-3\gamma(2)+\gamma(4))(1-1)=-\gamma(2)+2\gamma(3)-\gamma(4).$$
\end{exa}

\begin{exa}\label{exa_vert}
Let us compute the function $\mu_A$ on the cone considered in Examples \ref{trop_poly} and \ref{seq_poly}. 
By Corollary \ref{sp_case}, we have 
\begin{multline*}
\mu_A(\gamma)=(2-2)^2\gamma(1)+(4-1)(1+2+4-5)\gamma(2)+(4-1-2)(8+2-6)\gamma(4)-\gamma(2)+2\gamma(3)-\gamma(4)=\\=0\cdot\gamma(1)+5\gamma(2)+2\gamma(3)+3\gamma(4).
\end{multline*}
\end{exa}
\subsection{Example: degree 4 polynomials}
In this subsection we will compute the vertices of the Newton polytope of the Morse discriminant in the space of polynomials with support $A=\{1,2,3,4\}$ using Theorem \ref{support} and its special case mentioned in Corollary \ref{sp_case}.

To do this, we need to enumerate all possible tuples of combinatorial data arising from generic covectors in $(\R^4)^*$ and evaluate the support function on the corresponding cones. 

There are $5$ cases, we will treat each of them separately. The graphs of typical representatives $\gamma\in(\R^4)*$ (viewed as functions $\gamma\colon A\to\R$) of the corresponding cones are shown in Figure 3. The elements of the subsets $W$ are marked red. By comparing the slope of the interval $[(1,\gamma(1)),(3,\gamma(3))]$ with the slope of $[(2,\gamma(2)),(4,\gamma(4))]$ we distinguish the cases $4$ and $5$ (see Remark \ref{slopes_r2}).

\begin{itemize}
    \item[--] $W=\{1,4\}.$ In this case, we have $\gcd(w_0,w_1)=1,$ therefore the summand $C^0$ is equal to $0$ and does not depend on the sequence $M^0.$ The corresponding vertex is $(4,0,0,6).$
    \item[--] $W=\{1,2,3,4\}.$ Similarly to the previous case, the summands $C^j$ do not depend on the sequences $M^j$ and are equal to $0.$ The corresponding vertex is $(0,2,8,0).$
	\item[--] $W=\{1,3,4\}.$ In this case, $\gcd(w_0,w_1)=1,~\gcd(w_1,w_2)=1,$ so the summands $C^0$ and $C^1$ do not depend on the sequences $M^0$ and $M^1$ and are equal to $0.$ 
	The corresponding vertex is $(1,0,9,0).$
    \item[--] $W=\{1,2,4\},~M^1=(3,1).$ This case was considered in detail in Example \ref{exa_vert}, and the corresponding vertex is $(0,5,2,3).$
	\item[--] $W=\{1,2,4\},~M^1=(1,3).$ and the corresponding vertex is $(2,3,0,5).$
\end{itemize}
	\begin{center}
	\begin{tikzpicture}[scale=0.7]
	\draw[ultra thick] (0,0)--(3,0);
	\draw[ultra thick] (5,0)--(8,0);
	\draw[ultra thick] (10,0)--(13,0);
	\draw[thick] (0,0)--(0,2);
	\draw[thick] (3,0)--(3,3);
	\draw[ultra thick] (0,2)--(3,3);
	\draw[dashed] (1,0)--(1,1);
	\draw[dashed] (2,0)--(2,1.8);
	\filldraw[color=black, fill=red!70, thick](0,0) circle (0.1);
	\filldraw[color=black, fill=white, thick](1,0) circle (0.1);
	\filldraw[color=black, fill=red!70, thick](3,0) circle (0.1);
	\filldraw[color=black, fill=white, thick](2,0) circle (0.1);
	\filldraw[color=black, fill=black, thick](0,2) circle (0.1);
	\filldraw[color=black, fill=black, thick](1,1) circle (0.1);
	\filldraw[color=black, fill=black, thick](2,1.8) circle (0.1);
	\filldraw[color=black, fill=black, thick](3,3) circle (0.1);
	\draw[thick] (5,0)--(5,1);
	\draw[thick] (8,0)--(8,1.5);
	\draw[dashed] (6,0)--(6,3.5);
	\draw[dashed] (7,0)--(7,3);
	\draw[ultra thick] (5,1)--(6,3.5);
	\draw[ultra thick] (7,3)--(6,3.5);
	\draw[ultra thick] (7,3)--(8,1.5);
	\filldraw[color=black, fill=red!70, thick](5,0) circle (0.1);
	\filldraw[color=black, fill=red!70, thick](6,0) circle (0.1);
	\filldraw[color=black, fill=red!70, thick](7,0) circle (0.1);
	\filldraw[color=black, fill=red!70, thick](8,0) circle (0.1);
	\filldraw[color=black, fill=black, thick](5,1) circle (0.1);
	\filldraw[color=black, fill=black, thick](6,3.5) circle (0.1);
	\filldraw[color=black, fill=black, thick](7,3) circle (0.1);
	\filldraw[color=black, fill=black, thick](8,1.5) circle (0.1);
	\draw[thick] (10,0)--(10,1);
	\draw[thick] (13,0)--(13,1.5);
	\draw[dashed] (11,0)--(11,1.3);
	\draw[dashed] (12,0)--(12,3);
	\draw[ultra thick] (10,1)--(12,3);
	\draw[ultra thick] (12,3)--(13,1.5);
	\filldraw[color=black, fill=red!70, thick](10,0) circle (0.1);
	\filldraw[color=black, fill=white, thick](11,0) circle (0.1);
	\filldraw[color=black, fill=red!70, thick](12,0) circle (0.1);
	\filldraw[color=black, fill=red!70, thick](13,0) circle (0.1);
	\filldraw[color=black, fill=black, thick](10,1) circle (0.1);
	\filldraw[color=black, fill=black, thick](11,1.3) circle (0.1);
	\filldraw[color=black, fill=black, thick](12,3) circle (0.1);
	\filldraw[color=black, fill=black, thick](13,1.5) circle (0.1);
		\node[below] at (0,0) {${1}$};
	\node[below] at (1,0) {${2}$};
	\node[below] at (2,0) {${3}$};
	\node[below] at (3,0) {${4}$};
	\node[below] at (5,0) {${1}$};
	\node[below] at (6,0) {${2}$};
	\node[below] at (7,0) {${3}$};
	\node[below] at (8,0) {${4}$};
	\node[below] at (10,0) {${1}$};
	\node[below] at (11,0) {${2}$};
	\node[below] at (12,0) {${3}$};
	\node[below] at (13,0) {${4}$};
	\node[below,scale=0.8] at (1.5,-1) {\circled{1}};
	\node[below,scale=0.8] at (6.5,-1) {\circled{2}};
	\node[below,scale=0.8] at (11.5,-1) {\circled{3}};
	\end{tikzpicture}
	\begin{tikzpicture}[scale=0.7]
   	\draw[ultra thick] (15,0)--(18,0);
	\draw[thick] (18,0)--(18,2.5);
	\draw[thick] (15,0)--(15,1.5);
	\draw[dashed] (17,0)--(17,2);
	\draw[dashed] (16,0)--(16,3);
	\draw[ultra thick] (18,2.5)--(16,3);
	\draw[ultra thick] (16,3)--(15,1.5);
	\draw[ultra thick, dashed, violet] (15,1.5)--(17,2);
	\filldraw[color=black, fill=red!70, thick](15,0) circle (0.1);
	\filldraw[color=black, fill=white, thick](17,0) circle (0.1);
	\filldraw[color=black, fill=red!70, thick](16,0) circle (0.1);
	\filldraw[color=black, fill=red!70, thick](18,0) circle (0.1);
	\filldraw[color=black, fill=black, thick](18,2.5) circle (0.1);
	\filldraw[color=black, fill=black, thick](17,2) circle (0.1);
	\filldraw[color=black, fill=black, thick](16,3) circle (0.1);
	\filldraw[color=black, fill=black, thick](15,1.5) circle (0.1);
	\node[below] at (15,0) {${1}$};
	\node[below] at (16,0) {${2}$};
	\node[below] at (17,0) {${3}$};
	\node[below] at (18,0) {${4}$};
	\draw[ultra thick] (20,0)--(23,0);
	\draw[thick] (23,0)--(23,2.5);
	\draw[thick] (20,0)--(20,2);
	\draw[dashed] (22,0)--(22,1);
	\draw[dashed] (21,0)--(21,3);
	\draw[ultra thick] (23,2.5)--(21,3);
	\draw[ultra thick] (21,3)--(20,2);
	\draw[ultra thick, dashed, violet] (20,2)--(22,1);
	\filldraw[color=black, fill=red!70, thick](20,0) circle (0.1);
	\filldraw[color=black, fill=white, thick](22,0) circle (0.1);
	\filldraw[color=black, fill=red!70, thick](21,0) circle (0.1);
	\filldraw[color=black, fill=red!70, thick](23,0) circle (0.1);
	\filldraw[color=black, fill=black, thick](23,2.5) circle (0.1);
	\filldraw[color=black, fill=black, thick](22,1) circle (0.1);
	\filldraw[color=black, fill=black, thick](21,3) circle (0.1);
	\filldraw[color=black, fill=black, thick](20,2) circle (0.1);
	\node[below] at (20,0) {${1}$};
	\node[below] at (21,0) {${2}$};
	\node[below] at (22,0) {${3}$};
	\node[below] at (23,0) {${4}$};
	\node[below,scale=0.8] at (16.5,-1) {\circled{4}};
	\node[below,scale=0.8] at (21.5,-1) {\circled{5}};
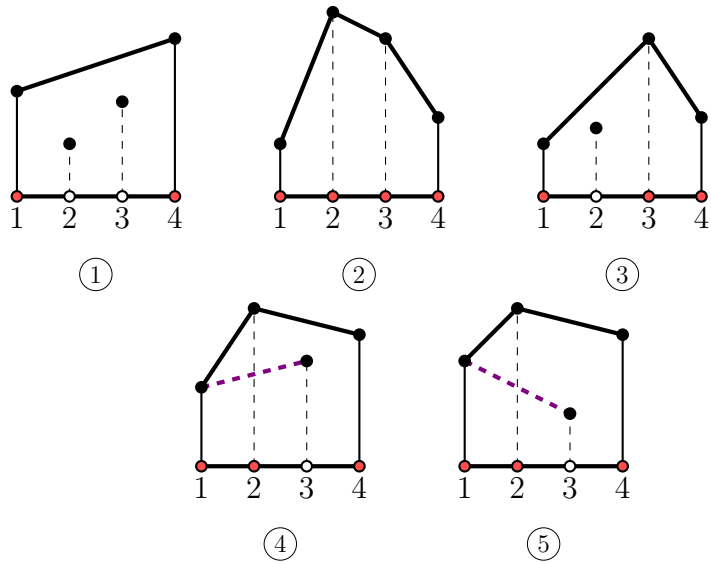
	\node[below] at (18,-2) {Figure 3. The $5$ vertices of the polytope $\mathcal M_A$.};
	\end{tikzpicture}
\end{center}
The polytope $\mathcal M_A$ is a polygon in $\R^4,$ and its image under the projection forgetting the first and the last coordinates is shown in Figure $4$ below. 
\begin{center}
	\begin{tikzpicture}[scale=0.4]
	\draw[ultra thick] (9,0)--(8,2);
	\draw[ultra thick] (8,2)--(2,5);
	\draw[ultra thick] (2,5)--(0,3);
	\draw[ultra thick] (0,3)--(0,0);
	\draw[ultra thick] (0,0)--(9,0);
	\filldraw[color=black, fill=blue!70, thick](0,0) circle (0.15);
	\filldraw[color=black, fill=blue!70, thick](8,2) circle (0.15);
	\filldraw[color=black, fill=blue!70, thick](9,0) circle (0.15);
	\filldraw[color=black, fill=blue!70, thick](2,5) circle (0.15);
    \filldraw[color=black, fill=blue!70, thick](0,3) circle (0.15);
	\node[left,scale=0.9] at (-0.2,0) {\circled{1}};
	\node[right,scale=0.9] at (8.2,2) {\circled{2}};
	\node[right,scale=0.9] at (9.2,0) {\circled{3}};
	\node[above right,scale=0.9] at (2,5) {\circled{4}};
	\node[left,scale=0.9] at (-0.2,3) {\circled{5}};
	\node[below] at (5,-2) {Figure 4. A projection of the polytope $\mathcal M_A$.};
	\end{tikzpicture}
\end{center}

\section{Proof of the Main Result}\label{the_proof}
\subsection{Step 1: the geometric interpretation of the problem}
The vertices of the polytope $\mathcal M_{A}$ are in 1-to-1 correspondence with the full-dimensional cones of its dual fan. The support function $\mu_A$ is linear on each of these cones, and its coefficients on the given cone are the coordinates of the corresponding vertex. 

As it was discussed in Remark \ref{homog}, the sought polytope lies in the intersection of the following two hyperplanes in $\R^{|A|}:$ the hyperplane  $\{e_0+\ldots+e_{|A|-1}=d_1\}$ and the hyperplane $\{a_0\cdot e_0+\ldots+a_{|A|-1}\cdot e_{|A|-1}=d_2\}$ for some $d_1,d_2.$ Therefore it suffices to compute $\mu_A$ only on the covectors $\gamma\in(\R^{|A|})^*$ with non-negative entries. 

Moreover, to find the coefficients of the function $\mu_A$ on its domains of linearity, it is enough to compute $\mu_{A}$ on rational, or, equivalently, on integer covectors supported at the corresponding vertices. The latter observation allows to use the following geometric interpretation of the main problem. Namely, we can use the same idea as, for instance, in the works \cite{AE},\cite{E3} and \cite{EK}. Let $\gamma$ be an integer covector with non-negative entries supported at a vertex of $\mathcal M_{A}.$ Alternatively it can be viewed as a function $\gamma\colon A\to\Z_{\geqslant 0}.$
Replacing the coefficients of $x^p,~p\in A$ of a polynomial $f(x),~\supp(f)=A,$ with polynomials of degrees $\gamma(p)$ in a new variable $t$ turns the Morse discriminant into a polynomial in $t.$ And since we can interpret the value $\mu_A(\gamma)$ as the number of roots of this univariate polynomial, the main problem of this paper can be reduced to the following one.

\begin{prb}
Let $\gamma\colon A\to \Z_{\geqslant 0}$ be an arbitrary function and $q_p,v_p,~p\in A$ be generic tuples of complex numbers. For how many values of the parameter $t\in \C$ is the polynomial $f_t(x)=\sum\limits_{p\in A}(q_p+v_pt^{\gamma(p)})x^p$ not Morse?
\end{prb}

This question was discussed in Example 1.1 of \cite{E3}, and the answer was obtained for a special case of a concave function $\gamma\colon\{1,\ldots,n\}\to\Z_{\geqslant 0}.$ Using a similar approach, we will obtain the answer for any function $\gamma\colon A\to\Z_{\geqslant 0}.$ 

We will represent the function $\mu_A$ as the solution to a system of $3$ equations, which we will deduce in Subsections \ref{Eq12}--\ref{sub_FPS}, and then solve those equations in Subsection \ref{sol_eq} to obtain the final answer.

\subsection{Step 2: the first two equations}\label{Eq12}

Given a function $\gamma \colon A\to \Z_{\geqslant 0},$ we consider the hypersurface $\mathcal H=\{f_t(x)-y\}\subset \CC^3.$ Let $\pi$ be the projection $\pi\colon\CC^3\to\CC^2,~~(x,y,t)\mapsto (y,t)$, and let $\mathcal A_1$, $\mathcal A_2$ and $2\mathcal A_1$ be the open multisingularity strata of its restriction $\pi\mid_{\mathcal H}$ to the hypersurface $\mathcal H$. The sets $\mathcal A_1, \mathcal A_2$ and $2\mathcal A_1$ consist of all the points $(y,t)\in\CC^2$ such that the equation $f_t(x)=y$ has exactly one root of multiplicity $2$, exactly one root of multiplicity $3$ and exactly two roots of multiplicity $2$, respectively. Due to Assumption \ref{gen_sing}, for the set $\tilde{A},$ the only strata of codimension $2$ are $\mathcal A_2$ and $2\mathcal A_1,$ while the strata of other singularities are of strictly higher codimension. 

Denote by $\Delta$ the Newton polytope of the polynomial $f_t(x)-y.$ Equivalently, $\Delta$ is the convex hull of the set $$\tilde{A}=\{(0,1,0)\}\cup\{(a_0,0,0)\}\cup\{(a_{|A|-1},0,0)\}\cup\{\big(p,0,\gamma(p)\big)\mid p\in A\}\subset\Z^3.$$ 

All the convex subdivisions of the interval $\conv(A)$ are in 1-to-1 correspondence with the subsets of the form $W=\{w_0 < w_1<\ldots < w_{k-1} < w_k\}\subset A$ with $w_0=a_0$ and $w_k=a_{|A|-1}.$ Moreover, every function $\gamma \colon A\to \Z_{\geqslant 0}$ defines a convex subdivision of the interval $\conv(A).$ Indeed, let $N$ be the Newton polygon of the polynomial $f_t(x)=\sum\limits_{p\in A}(q_p+v_pt^{\gamma(p)})x^p.$ Then the corresponding subset $W$ consists of all the points $p\in A$ such that the point $(p,\gamma(p))$ is a vertex of $N.$

\begin{lemma}\label{eq12}
	Let $\gamma\colon A\to \Z_{\geqslant 0}$ be a function and $W$ be the corresponding convex subsivision of the interval $\conv(A).$ Then, the following equalities hold:
	\begin{equation}\label{e12}
	\begin{cases}
	\chi(\mathcal A_1)+2|2\mathcal A_1|+2|\mathcal A_2|=-\area(N)\\
	|\mathcal A_2|=\area(N)-\gamma(w_0)-\gamma(w_k).
	\end{cases}
	\end{equation} 
\end{lemma}

\begin{proof}
The first equation follows from the Kouchnirenko--Bernstein--Khovanskii formula (see \cite{K1} for details). Indeed, the Euler characteristic of the hypersurface $\mathcal H\subset\CC^3$ is equal to the lattice volume of $\Delta$. On the other hand, let $D\subset\CC^2$ be the set of the critical values of the projection $\pi,$ that is, the union of the strata $\mathcal A_1,\mathcal A_2$ and $2\mathcal A_1$. Then, counting $\chi(\mathcal H)$ fiberwise and using the additivity property of the Euler characteristic, we obtain: 
	\begin{multline*}
	\chi(\mathcal H)=(w_k-w_0)\chi(\CC^2\setminus D)+(w_k-w_0-1)\chi(\mathcal A_1)+(w_k-w_0-2)\chi(\mathcal A_2)+(w_k-w_0-2)\chi(2\mathcal A_1)=\\(w_k-w_0)\chi(\CC^2)-\chi(\mathcal A_1)-2|\mathcal A_2|-2|2\mathcal A_1|=\\0-\chi(\mathcal A_1)-2|\mathcal A_2|-2|2\mathcal A_1|. 
	\end{multline*}	
	Thus, $\chi(\mathcal A_1)+2|\mathcal A_1|+2|\mathcal A_2|=-\vol(\Delta)=-\area(N),$ and this is the first of the desired equations.
	
	The second equation follows directly from the formula for the support function of the secondary polytope (see Formula (1.4) in Section 7.1D of \cite{GKZ}).
\end{proof}
\begin{exa}
	If $A=[1,n]\cap\Z$ and $\gamma$ is a concave function, then the equations (\ref{e12}) are of the following form: 
	$$
	\begin{cases}
	\chi(\mathcal A_1)+2|2\mathcal A_1|+|\mathcal A_2|=-\gamma(1)-\gamma(n)-\sum\limits_{m=2}^{n-1}2\gamma(m)\\
	|\mathcal A_2|=\sum\limits_{m=2}^{n-1}2\gamma(m).
	\end{cases}
	$$
\end{exa}

\subsection{Step 3: the third equation}

The equations (\ref{e12}) given by Lemma \ref{eq12} do not suffice to compute $|2\mathcal A_1|.$ To obtain the final equation, we will use the technique developed in \cite{V}.

Namely, let $\mathcal{C}\subset\CC^3$ be the complete intersection curve defined by the system \begin{equation}\label{def_C}
\{f_t(x)-y=x\frac{\partial f_t(x)}{\partial x}=0\}.
\end{equation} 
In other words, the curve $\mathcal C$ is the set of the critical points of the projection $\pi\colon \mathcal H\to\CC^2.$ For sufficiently generic polynomials $g_1,g_2$ with the same Newton polytope $\Delta,$ the Newton polygon of the projection of the complete intersection $\{g_1=g_2=0\}$ is equal to the fiber polytope $P=\int_{\pi}\Delta$ of $\Delta$ (see \cite{GP}, \cite{EK}). The coefficients in the system (\ref{def_C}) are related to each other, since the second equation is a partial derivative of the first one. So, we have to verify that the system (\ref{def_C}) satisfies certain genericity conditions. Namely, following Section 4.2 of \cite{GP}, we need to check that for every non-horizontal (e.g. not parallel to the plane of projection) edge $e\subset\Delta,$ the truncations of the equations in (\ref{def_C}) to $e$ do not have roots in $\CC.$ 

The non-horizontal edges $e\subset \Delta$ are of the following three types: 
\begin{itemize}
	\item[--] the edge $[(w_0,0,0),(w_k,0,0)];$ then the truncated system is $\{f_0=x\frac{\partial f_0}{\partial x}=0\},$ where $f_0=\sum\limits_{p\in A} q_px^p.$ 
	For generic $q_p$ this system does not have a root in $\CC.$
	\item[--] the edges $[(0,1,0),(w_0,0,0)],~[(0,1,0),(w_k,0,0)]$ and $[(0,1,0),(w_j,\gamma(w_j))],0\leqslant j\leqslant k;$ the second polynomial in the truncated system is a monomial, so, the system has no roots in $\CC.$
	\item[--] the edges $[(w_j,0,\gamma(w_j)),(w_{j+1},0,\gamma(w_{j+1}))],~0\leqslant j\leqslant k-1;$
	the truncated system is $\{x^{w_j}t^{\gamma(w_j)}+x^{w_{j+1}}t^{\gamma(w_{j+1})}=w_jx^{w_j}t^{\gamma(w_j)}+w_{j+1}x^{w_{j+1}}t^{\gamma(w_{j+1})}=0\},$ which implies the equation $(w_{j+1}-w_j)x^{w_{j+1}}t^{\gamma(w_{j+1})}=0.$ Thus the system has no roots in $\CC.$ 
\end{itemize}

So, the Newton polytope of the curve $D$ is indeed the fiber polygon $P=\int_{\pi}\Delta.$ 
We can find the desired third relation by comparing the Euler characteristic $\chi(D)$ with the one of its generic perturbation which is a smooth Newton non-degenerate curve $Y$ with the same Newton polygon $P.$ The latter is known to be equal to $-\vol(P),$ by \cite{K1}. 

Let us note that the complete intersection $\mathcal C$ is degenerate with respect to the polytope $\Delta,$ because its closure intersects two one-dimensional orbits of the toric variety $X_{\Delta}.$ Those are precisely the orbits corresponding to the edges $[(w_0,0,0),(w_0,0,\gamma(w_0))]$ and $[(w_k,0,0),(w_k,0,\gamma(w_k))].$ So, we will have to deal with the contribution of the corresponding intersection points to the Euler characteristic of the curve $D.$ 

One can easily check that the closure of the curve $\mathcal C$ transversally intersects the $2$-dimensional orbits of the toric variety $X_{\Delta}$ corresponding to the facets of $\Delta$ not containing the edges $[(w_0,0,0),(w_0,0,\gamma(w_0))]$ and $[(w_k,0,0),(w_k,0,\gamma(w_k))].$ The singularities of the closure of the curve $D$ arising from these facets as well as their contribution to the Euler characteristic of the curve $D$ were studied in \cite{V} (see Section 2 for an overview).

Thus, the closure of the curve $D$ in the toric surface $X_P$ has the following 4 types of singularities: 

\begin{itemize}
	\item the set $\mathcal S_1$ consisting of $|2\mathcal A_1|$ transversal self-intersections;
	\item the set $\mathcal S_2$ consisting of $|\mathcal A_2|$ simple cusps; 
	\item the set  $\mathcal S_{hor}$ of the singularities arising from the horizontal edges of $\Delta;$
	\item the set $\mathrm{FPS}$ (from {\it forking--path singularities}) arising from the facets of $\Delta.$
\end{itemize}

Using this information, we can prove the following statement. 
\begin{utver}\label{Rel3}
	In the same notation as above, we have the following equality: 
	\begin{multline}\label{rel3}
	\chi(\mathcal A_1)-|\mathcal A_2|=\\=-\vol(P)-\sum_{s\in \mathcal S_{hor}}\chi([\text{Milnor fiber of~} s]\cap\CC^2)-\sum_{s\in \mathrm{FPS}}\chi([\text{Milnor fiber of~} s]\cap\CC^2).
	\end{multline}
\end{utver}
\begin{proof}
	Indeed, by additivity of Euler characteristic, we have the following equality: 
	\begin{multline}\label{comp_chi1}
	\chi(Y)=\chi(D)-|\mathcal A_2|-|2\mathcal A_1|+\sum_{s\in\mathcal S_1} (b(s)-2\delta(s))+\sum_{s\in\mathcal S_2} (b(s)-2\delta(s))+\\+\sum_{s\in\mathcal S_{hor}}\chi([\text{Milnor fiber of~} s]\cap\CC^2)+\sum_{s\in\mathrm{FPS}}\chi([\text{Milnor fiber of~} s]\cap\CC^2),
	\end{multline}
	where $b(s)$ is the number of branches passing through the singular point $s$ and $\delta(s)$ is the $\delta$-invariant of $s$.
	
	Formula (\ref{comp_chi1}) can be interpreted as follows: we puncture small neighborhoods of the points in $\mathcal S_1\cup\mathcal S_2$ and replace them with the Milnor fibers of the corresponding singularities. By Milnor formula, the latter are of Euler characteristic $b(s)-\delta(s).$ Finally, we add the Milnor fibers of the singularities in $\mathrm{FPS}\cup\mathcal S_{hor}.$
	
	Substituting the values $b(s)$ and $\delta(s)$ for cusps and nodes and using the equality $\chi(D)=\chi(\mathcal A_1)+|2\mathcal A_1|+|\mathcal A_2|$, we obtain the following equality: 
	\begin{multline}\label{comp_chi2}
	\chi(Y)=\chi(\mathcal A_1)+|2\mathcal A_1|+|\mathcal A_2|-|\mathcal A_2|-|2\mathcal A_1|+|2\mathcal A_1|(2-2)+|\mathcal A_2|(1-2)+\\+\sum_{s\in \mathcal S_{hor}}\chi([\text{Milnor fiber of~} s]\cap\CC^2)+\sum_{s\in \mathrm{FPS}}\chi([\text{Milnor fiber of~} s]\cap\CC^2)=\\=
	\chi(\mathcal A_1)-|\mathcal A_2|+\sum_{s\in \mathcal S_{hor}}\chi([\text{Milnor fiber of~} s]\cap\CC^2)+\sum_{s\in \mathrm{FPS}}\chi([\text{Milnor fiber of~} s]\cap\CC^2)
	\end{multline}
	Finally, by \cite{K1}, we have $\chi(Y)=-\vol(P).$ After using this equality and rearranging the terms in (\ref{comp_chi2}), we obtain the desired equation. 
\end{proof}

The equations in (\ref{e12}) combined with (\ref{rel3}) suffice to compute $|2\mathcal A_1|.$ So, our next goal is to compute the right-hand side of (\ref{rel3}) in terms of $\gamma.$ Moreover, from \cite{V}, it follows that unlike the first two equations, the third one does not only employ the terms $\gamma(p)$ such that $(p,0,\gamma(p))$ is a vertex of $\Delta$, but also the other ones, because so do the formulas for the Euler characteristic of the Milnor fiber of the singularities $s\in\mathrm{FPS}.$ In other words, to find $|2\mathcal A_1|,$ one needs more information than the convex subdivision of the interval $\conv(A)$ defined by the covector $\gamma.$ 
\subsection{Step 4: The area of the fiber polygon \texorpdfstring{$P$}{P}}
This subsection is devoted to computing the first of the three summands in the right-hand side of (\ref{rel3}), namely, the area of the polygon $P=\int_{\pi}\Delta.$ Let $\gamma\colon A\to\Z_{\geqslant 0}$ be a function satisfying Assumptions \ref{slopes} and \ref{roots}, and let $W=\{w_0,\ldots,w_k\}\subset A$ and $Z=(z_1,\ldots,z_k)$ be the corresponding pieces of combinatorial data introduced in Definition \ref{def_comb_data}.

Let us first describe the polytope $P$. The polytope $\Delta$ that we start with is the convex hull of the set $\{(0,1,0)\}\cup\{(a_0,0,0)\}\cup\{(a_{|A|-1},0,0)\}\cup\{\big(p,0,\gamma(p)\big)\mid p\in A\}\subset\Z^3.$ 

\begin{exa}\label{Delta}
Consider the set $A=\{-3,-1,1,2,4\}$ and $\gamma=(3,5,2,5,1)$ from Example \ref{comb_data}. We have $W=\{-3,-1,2,4\}$ and $Z=(1,0,2).$ The polytope $\Delta$ is shown in Figure 5 below. 
\begin{center}
\begin{tikzpicture}[scale=0.45]
\draw[thick,gray,->] (7,-2)--(16,16);
\draw[thick, gray,dashed] (11,6)--(1,11);
\draw[thick, gray,->] (1,11)--(-1,12);
\draw[thick,gray,->] (11,6)--(9,12);
\draw[thick] (8,0)--(15,14);
\draw[thick] (8,0)--(2,3);
\draw[thick] (2,3)--(0,9);
\draw[thick] (0,9)--(3,15);
\draw[thick] (3,15)--(13,15);
\draw[thick] (15,14)--(13,15);
\draw[thick] (10,9)--(15,14);
\draw[thick] (10,9)--(13,15);
\draw[thick] (10,9)--(3,15);
\draw[thick] (10,9)--(0,9);
\draw[thick] (10,9)--(2,3);
\draw[thick] (10,9)--(8,0);
\filldraw[color=black, fill=red!70, thick](8,0) circle (0.2);
\filldraw[color=black, fill=red!70, thick](10,4) circle (0.2);
\filldraw[color=black, fill=red!70, thick](13,10) circle (0.2);
\filldraw[color=black, fill=red!70, thick](15,14) circle (0.2);
\filldraw[color=black, fill=black, thick](11,6) circle (0.2);
\filldraw[color=black, fill=black, thick](10,9) circle (0.2);
\filldraw[color=black, fill=white, thick](12,8) circle (0.2);
\node[right,scale=0.75] at (8,0) {$(-3,0,0)$};
\node[right,scale=0.75] at (10,4) {$(-1,0,0)$};
\node[right,scale=0.75] at (11,6) {$(0,0,0)$};
\node[right,scale=0.75] at (13,10) {$(2,0,0)$};
\node[right,scale=0.75] at (15,14) {$(4,0,0)$};
\node[right,scale=0.75] at (10,9) {$(0,1,0)$};
\node[left,scale=0.75] at (2,3) {$(-3,0,3)$};
\node[left,scale=0.75] at (0,9) {$(-1,0,5)$};
\node[above,scale=0.75] at (3,15) {$(2,0,5)$};
\node[above,scale=0.75] at (13,15) {$(4,0,1)$};
\node[right,gray,scale=0.75] at (16,16) {$e_1$};
\node[right,gray,scale=0.75] at (9,12) {$e_2$};
\node[above,gray,scale=0.75] at (-1,12) {$e_3$};
\node[below] at (8,-2) {Figure 5. The polytope $\Delta.$};
\end{tikzpicture}
\end{center}
\end{exa}
The projection that we are interested in forgets the first coordinate. The points of the fiber polytope $P$ are integrals of monotone paths in the polytope $\Delta$ connecting the upper and the lower horizontal faces of $\Delta.$ The vertices of $P$ are in 1-to-1 correspondence with monotone paths going along the edges of the polytope $\Delta$ arising from boundaries of projections of the polytope $\Delta$ along the directions parallel to the plane containing the fiber polytope $P.$ 

\begin{center}
	\begin{tikzpicture}[scale=0.3]
	\draw[thick,gray,->] (7,-2)--(16,16);
	\draw[thick, gray,dashed] (11,6)--(1,11);
	\draw[thick, gray,->] (1,11)--(-1,12);
	\draw[thick,gray,->] (11,6)--(9,12);
	\draw[thick] (8,0)--(15,14);
	\draw[thick] (8,0)--(2,3);
	\draw[thick] (2,3)--(0,9);
	\draw[thick] (0,9)--(3,15);
	\draw[thick] (3,15)--(13,15);
	\draw[thick] (15,14)--(13,15);
	\draw[thick] (10,9)--(15,14);
	\draw[thick] (10,9)--(13,15);
	\draw[thick] (10,9)--(3,15);
	\draw[thick] (10,9)--(0,9);
	\draw[thick] (10,9)--(2,3);
	\draw[thick] (10,9)--(8,0);
    \draw[ultra thick, violet] (13,15)--(3,15)--(0,9)--(2,3);
	\draw[ultra thick, orange] (15,14)--(10,9)--(8,0);
	\draw[ultra thick,cyan] (8,0)--(15,14);
	\draw[ultra thick, teal] (13,15)--(10,9)--(2,3);
	\filldraw[color=black, fill=black, thick](13,15) circle (0.2);
	\filldraw[color=black, fill=black, thick](3,15) circle (0.2);
	\filldraw[color=black, fill=black, thick](0,9) circle (0.2);
	\filldraw[color=black, fill=black, thick](2,3) circle (0.2);
	\filldraw[color=black, fill=red!70, thick](8,0) circle (0.3);
	\filldraw[color=black, fill=red!70, thick](10,4) circle (0.3);
	\filldraw[color=black, fill=red!70, thick](13,10) circle (0.3);
	\filldraw[color=black, fill=red!70, thick](15,14) circle (0.3);
	\filldraw[color=black, fill=black, thick](11,6) circle (0.3);
	\filldraw[color=black, fill=black, thick](10,9) circle (0.3);
	\filldraw[color=black, fill=white, thick](12,8) circle (0.3);
	\node[right,gray,scale=0.8] at (16,16) {$e_1$};
	\node[right,gray,scale=0.8] at (9,12) {$e_2$};
	\node[above,gray,scale=0.8] at (-1,12) {$e_3$};
	\draw[thick,gray,->] (21,-1)--(21,15);
	\draw[thick, gray, ->] (22,6)--(17,6);
	\draw[ultra thick, cyan] (21,0)--(21,14);
	\draw[ultra thick, orange] (21,14)--(18,6)--(21,0);
	\filldraw[color=black, fill=red!70, thick](21,0) circle (0.25);
	\filldraw[color=black, fill=red!70, thick](21,4) circle (0.25);
	\filldraw[color=black, fill=red!70, thick](21,10) circle (0.25);
	\filldraw[color=black, fill=red!70, thick](21,14) circle (0.25);
	\filldraw[color=black, fill=black, thick](21,6) circle (0.25);
	\filldraw[color=black, fill=black, thick](18,6) circle (0.25);
	\filldraw[color=black, fill=black, thick](21,6) circle (0.25);
	\filldraw[color=black, fill=white, thick](21,8) circle (0.25);
	\node[right,gray,scale=0.8] at (21,15) {$e_1$};
	\node[left,gray,scale=0.8] at (17,6) {$e_2$};
	\draw[thick,gray,->] (27,-1)--(27,15);
	\draw[thick, gray, ->] (28,6.25)--(24,5.25);
	\draw[thick, gray, dashed] (28,6)--(25,6);
	\draw[thick, gray] (28,6)--(27,6);
	\draw[thick, gray, ->] (25,6)--(23.5,6);
	\draw[thick] (27,0)--(27,14);
	\draw[thick] (27,14)--(25,5.25)--(27,0);
	\draw[thick] (27,14)--(26,14);
	\draw[thick] (27,0)--(26,0);
	\draw[ultra thick, teal] (26,14)--(25,5.5)--(26,0);
	\filldraw[color=black, fill=black, thick](26,14) circle (0.15);
	\filldraw[color=black, fill=black, thick](26,0) circle (0.15);
	\filldraw[color=black, fill=red!70, thick](27,0) circle (0.25);
	\filldraw[color=black, fill=red!70, thick](27,4) circle (0.25);
	\filldraw[color=black, fill=red!70, thick](27,10) circle (0.25);
	\filldraw[color=black, fill=red!70, thick](27,14) circle (0.25);
	\filldraw[color=black, fill=black, thick](25,5.5) circle (0.25);
	\filldraw[color=black, fill=black, thick](27,6) circle (0.25);
	\filldraw[color=black, fill=white, thick](27,8) circle (0.25);
	\node[right,gray,scale=0.8] at (27,15) {$e_1$};
	\node[below left,gray,scale=0.8] at (24.3,5.25) {$e_2$};
	\node[above,gray,scale=0.8] at (23.5,6) {$e_3$};
    \end{tikzpicture}
    \begin{tikzpicture}[scale=0.3]
    \draw[thick,gray,->] (7,-2)--(16,16);
    \draw[thick, gray,dashed] (11,6)--(1,11);
    \draw[thick, gray,->] (1,11)--(-1,12);
    \draw[thick,gray,->] (11,6)--(9,12);
    \draw[thick] (8,0)--(15,14);
    \draw[thick] (8,0)--(2,3);
    \draw[thick] (2,3)--(0,9);
    \draw[thick] (0,9)--(3,15);
    \draw[thick] (3,15)--(13,15);
    \draw[thick] (15,14)--(13,15);
    \draw[thick] (10,9)--(15,14);
    \draw[thick] (10,9)--(13,15);
    \draw[thick] (10,9)--(3,15);
    \draw[thick] (10,9)--(0,9);
    \draw[thick] (10,9)--(2,3);
    \draw[thick] (10,9)--(8,0);
    \draw[ultra thick, magenta] (13,14.9)--(3.15,14.9)--(10,9.15)--(10,9)--(2,3);
    \draw[ultra thick, blue] (13,15.1)--(3.1,15.1)--(3,14.8)--(10,8.9)--(10,9)--(0,9)--(2,3);
   \filldraw[color=black, fill=black, thick](13,15) circle (0.2);
   \filldraw[color=black, fill=black, thick](3,15) circle (0.2);
   \filldraw[color=black, fill=black, thick](0,9) circle (0.2);
   \filldraw[color=black, fill=black, thick](2,3) circle (0.2);
    \filldraw[color=black, fill=red!70, thick](8,0) circle (0.3);
    \filldraw[color=black, fill=red!70, thick](10,4) circle (0.3);
    \filldraw[color=black, fill=red!70, thick](13,10) circle (0.3);
    \filldraw[color=black, fill=red!70, thick](15,14) circle (0.3);
    \filldraw[color=black, fill=black, thick](11,6) circle (0.3);
    \filldraw[color=black, fill=black, thick](10,9) circle (0.3);
    \filldraw[color=black, fill=white, thick](12,8) circle (0.3);
    \node[right,gray,scale=0.8] at (16,16) {$e_1$};
    \node[right,gray,scale=0.8] at (9,12) {$e_2$};
    \node[above,gray,scale=0.8] at (-1,12) {$e_3$};
    \draw[thick,gray,->] (25,-1)--(25,15);
    \draw[thick, gray, ->] (22,6)--(19,6);
    \draw[thick, gray] (28,6)--(25,6);
    \draw[thick, gray,->] (28,8)--(19,2);
    \draw[thick, gray, dashed] (25,6)--(22,6);
    \draw[thick] (25,0)--(25,14);
    \draw[thick] (23,13)--(22,4)--(23,-1);
    \draw[thick] (25,14)--(22,4)--(25,0);
    \draw[thick] (25,14)--(23,13);
    \draw[thick] (25,0)--(23,-1);
    \draw[ultra thick, magenta] (23,13)--(22,7.5)--(22,4)--(23,-1);
    \filldraw[color=black, fill=black, thick](23,13) circle (0.15);
    \filldraw[color=black, fill=black, thick](23,-1) circle (0.15);
     \filldraw[color=black, fill=black, thick](22,7.5) circle (0.15);
    \filldraw[color=black, fill=red!70, thick](25,0) circle (0.25);
    \filldraw[color=black, fill=red!70, thick](25,4) circle (0.25);
    \filldraw[color=black, fill=red!70, thick](25,10) circle (0.25);
    \filldraw[color=black, fill=red!70, thick](25,14) circle (0.25);
    \filldraw[color=black, fill=black, thick](22,4) circle (0.25);
    \filldraw[color=black, fill=black, thick](25,6) circle (0.25);
    \filldraw[color=black, fill=white, thick](25,8) circle (0.25);
    \node[right,gray,scale=0.8] at (25,15) {$e_1$};
    \node[below,gray,scale=0.8] at (19,2) {$e_2$};
    \node[above,gray,scale=0.8] at (19,6) {$e_3$};
    \draw[thick,gray,->] (38,-1)--(38,15);
    \draw[thick, gray, ->] (42,7)--(30,4);
    \draw[thick, gray] (42,6)--(38,6);
    \draw[thick, gray, dashed] (38,6)--(34,6);
    \draw[thick, gray,->] (34,6)--(30,6);
    \draw[thick] (38,0)--(38,14);    
    \draw[thick] (36,13.5)--(34,5)--(35,-0.25);
    \draw[thick] (38,14)--(34,5)--(38,0);
    \draw[thick] (38,14)--(36,13.5);
    \draw[thick] (38,0)--(35,-0.25);
    \draw[ultra thick, blue] (36,13.5)--(34,9)--(34,5)--(34,2)--(35,-0.25);
    \filldraw[color=black, fill=black, thick](36,13.5) circle (0.15);
    \filldraw[color=black, fill=black, thick](34,9) circle (0.15);
    \filldraw[color=black, fill=black, thick](35,-0.25) circle (0.15);
    \filldraw[color=black, fill=black, thick](34,2) circle (0.15);
    \filldraw[color=black, fill=red!70, thick](38,0) circle (0.25);
    \filldraw[color=black, fill=red!70, thick](38,4) circle (0.25);
    \filldraw[color=black, fill=red!70, thick](38,10) circle (0.25);
    \filldraw[color=black, fill=red!70, thick](38,14) circle (0.25);
    \filldraw[color=black, fill=black, thick](34,5) circle (0.25);
    \filldraw[color=black, fill=black, thick](38,6) circle (0.25);
    \filldraw[color=black, fill=white, thick](38,8) circle (0.25);
    \node[right,gray,scale=0.8] at (25,15) {$e_1$};
    \node[below,gray,scale=0.8] at (30,4) {$e_2$};
    \node[above,gray,scale=0.8] at (30,6) {$e_3$};
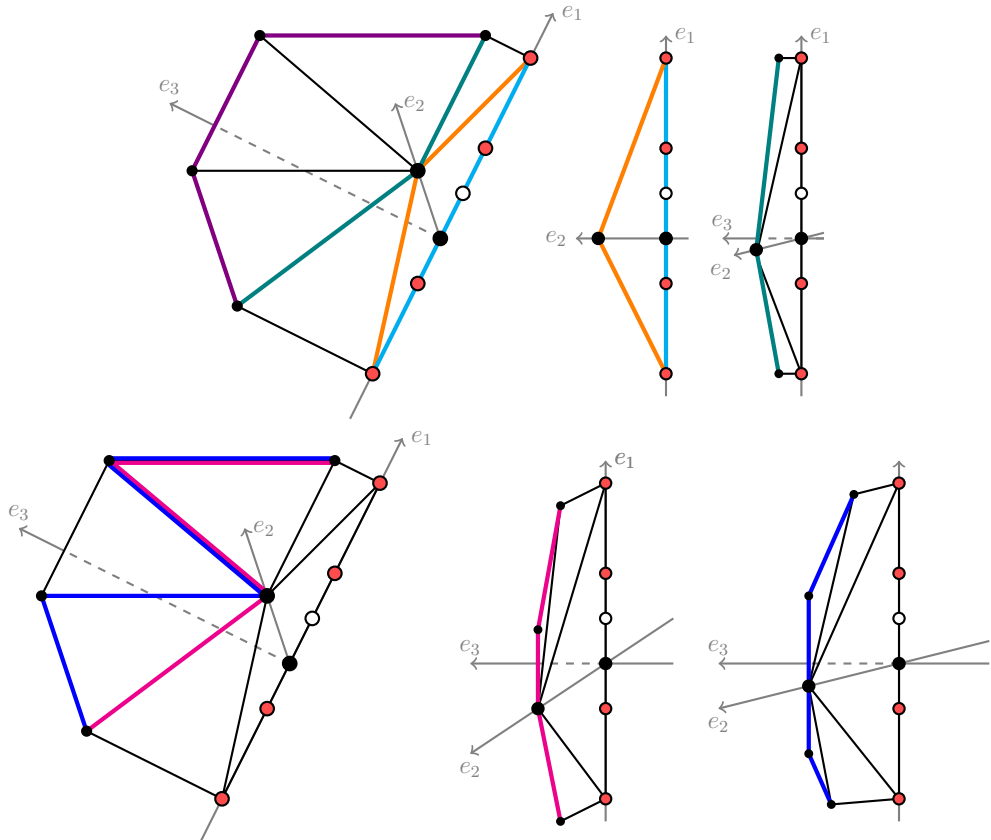
    \node[below] at (17,-2) {Figure 6. Describing the vertices of $\Delta.$};
    \end{tikzpicture}
\end{center}
\begin{exa}\label{path_vert_exa}
	Let us describe the paths corresponding to the vertices of the fiber polygon $P$ of the polytope $\Delta$ from Example \ref{Delta}. All these paths are shown in Figure 6. Each of these paths is drawn in colour on the polytope $\Delta.$ The most trivial example is the path shown in purple. It goes along the edges of the polygon $N,$ the base of the pyramid $\Delta,$ and can be seen on the boundary of the image under the projection forgetting the second coordinate. All the other paths are shown in color together with the correponding projections of $\Delta.$ 
\end{exa}

To compute the coordinates of the point in $P$ associated to the given path $$\Psi(s)=(s,\psi_1(s),\psi_2(s)),$$ we project this path onto the planes $\{e_2=0\}$ and $\{e_3=0\}$ and compute the definite integrals of the resulting functions in $s$. 

\begin{exa}\label{vert_coord_exa}
Let us find the vertex of $P$ corresponding to one of the paths shown in Example \ref{path_vert_exa}. We will work with the path drawn in blue in Figure 6. The coordinates of this vertex are exactly the areas of the regions shown below in Figure 7. The area of the first region is equal to the area of the base $N\subset\Delta$ minus the area of the triangle $\conv(\{(0,0),(w_1,\gamma(w_1)),(w_2,\gamma(w_2))\}).$ Thus, the first coordinate of the sought vertex is equal to: 
\begin{multline*}
(w_1-w_0)\gamma(w_0)+(w_2-w_0)\gamma(w_1)+(w_3-w_1)\gamma(w_2)+(w_3-w_2)\gamma(w_3)-(w_2\gamma(w_1)-w_1\gamma(w_2))\\=(w_1-w_0)\gamma(w_0)-w_0\gamma(w_1)+w_3\gamma(w_2)+(w_3-w_1)\gamma(w_3).
\end{multline*}
The area of the second region is equal to $w_2-w_1.$ 

\begin{center}
\begin{tikzpicture}[scale=0.6]
\draw[thick, gray, ->] (3,-2)--(3,6);
\draw[thick, gray, ->] (-1,0)--(8,0);
\filldraw[color=white, fill=blue!20, thick](0,0)--(0,3)--(2,5)--(3,0)--(5,5)--(7,1)--(7,0)--(0,0);
\draw[thick] (0,0)--(7,0);
\draw[thick] (0,0)--(0,3)--(2,5)--(5,5)--(7,1)--(7,0);
\draw[ultra thick, blue] (0,3)--(2,5)--(3,0)--(5,5)--(7,1);
\filldraw[color=black, fill=black, thick](0,3) circle (0.1);
\filldraw[color=black, fill=black, thick](2,5) circle (0.1);
\filldraw[color=black, fill=black, thick](5,5) circle (0.1);
\filldraw[color=black, fill=black, thick](7,1) circle (0.1);
\filldraw[color=black, fill=red!70, thick](0,0) circle (0.15);
\filldraw[color=black, fill=red!70, thick](2,0) circle (0.15);
\filldraw[color=black, fill=red!70, thick](5,0) circle (0.15);
\filldraw[color=black, fill=red!70, thick](7,0) circle (0.15);
\filldraw[color=black, fill=black, thick](3,0) circle (0.15);
\filldraw[color=black, fill=white, thick](4,0) circle (0.15);
\node[right,gray,scale=0.8] at (3,6) {$e_3$};
\node[below,gray,scale=0.8] at (8,0) {$e_1$};
\draw[thick, gray, ->] (9,0)--(18,0);
\filldraw[color=white, fill=blue!20, thick](12,0)--(13,3)--(15,0)--(17,0)--(12,0);
\draw[thick, gray, ->] (13,-2)--(13,4);
\draw[thick] (10,0)--(17,0);
\draw[thick] (10,0)--(13,3)--(17,0);
\draw[ultra thick, blue] (10,0)--(12,0)--(13,3)--(15,0)--(17,0);
\filldraw[color=black, fill=black, thick](13,3) circle (0.1);
\filldraw[color=black, fill=red!70, thick](10,0) circle (0.15);
\filldraw[color=black, fill=red!70, thick](12,0) circle (0.15);
\filldraw[color=black, fill=red!70, thick](15,0) circle (0.15);
\filldraw[color=black, fill=red!70, thick](17,0) circle (0.15);
\filldraw[color=black, fill=black, thick](13,0) circle (0.15);
\filldraw[color=black, fill=white, thick](14,0) circle (0.15);
\node[right,gray,scale=0.8] at (13,4) {$e_2$};
\node[below,gray,scale=0.8] at (18,0) {$e_1$};
\node[right, scale=0.7] at (13.2,3) {$(0,1)$};
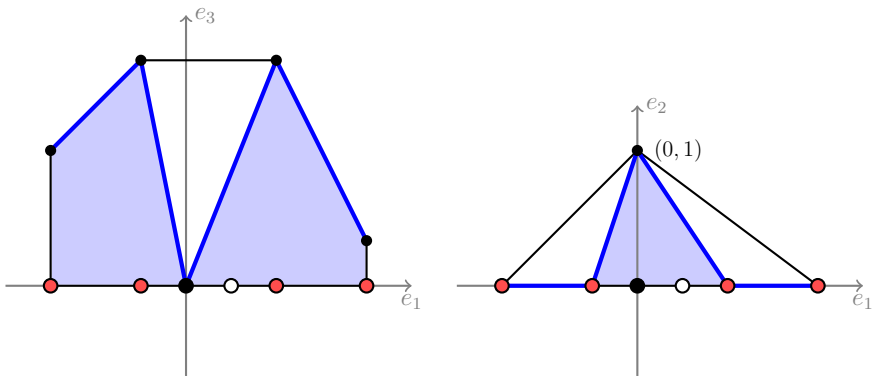
\node[below] at (8,-2) {Figure 7. The projections of the path $\Psi(s).$};
\end{tikzpicture}
\end{center}
\end{exa}

To describe the polygon $P$ for a covector $\gamma,$ we need to know the subdivision $W=(w_0,\ldots,w_k)\subset A,$ and the sequence $Z=(z_1,\ldots,z_k)$ (see Definition \ref{def_comb_data}). Recall that by $N$ we denote the base of the pyramid $\Delta.$ Let us introduce some notation. For $0\leqslant i<k,$ denote by $T_i$ the triangle $\conv(\{(0,0),(w_i,\gamma(w_i)),(w_{i+1},\gamma(w_{i+1}))\}),$ and set $d_i=w_{i+1}-w_i.$ The (oriented) lattice area $S_i=\area(T_i)$ is equal to $w_{i+1}\gamma(w_i)-w_i\gamma(w_{i+1}).$ We define the sequence $R_i, 0\leqslant i\leqslant k$ of sets as follows: 
$$R_0=\begin{cases}
N\cup\conv(\{(0,0),(w_0,\gamma(w_0)),(w_0,0)\}), \text{~if~}|w_0|<|w_k|,\\
N\cup\conv(\{(0,0),(w_k,\gamma(w_k)),(w_k,0)\}), \text{~if~}|w_0|>|w_k|,
\end{cases}$$
and for $i>0,$ we have $R_{i}=R_{i-1}\ominus T_{z_i},$ where $\ominus$ stands for the symmetric difference of sets, i.e., $E\ominus F=(E\cup F)\setminus(E\cap F).$

\begin{utver}\label{descr_P}
In the same notation as above, the fiber polygon $P$ is a union of rectangular trapezoids with bases $\Sigma_0,\ldots,\Sigma_k,$ sitting on top of each other. We have that $\length(\Sigma_i)=\area(R_i),$ and the distance between $\Sigma_{i+1}$ and $\Sigma_{i}$ is equal to $d_{z_i}.$
\end{utver}

\begin{exa}
Let us construct the fiber polytope for the set $A=\{-3,-1,1,2,4\}$ and the covector $\gamma=(3,5,5,2,1)$ from Examples \ref{Delta},\ref{path_vert_exa} and \ref{vert_coord_exa}. We have $W=\{w_0,w_1,w_2,w_3\}=\{-3,-1,2,4\}$ and $Z=(1,0,2).$ The triangles $T_i$ are shown in the Figure 8 below, as well as the sought fiber polytope $P.$ Each of the vertices of $P$ is shown in the same color as the corresponding path shown in Example \ref{path_vert_exa}. Using the definitions above, we obtain: 
 
$$\begin{cases}
R_0=N\\
R_1=R_0\ominus T_1=R_0\setminus T_1\\
R_2=R_1\ominus T_0=R_1\setminus T_0\\
R_3=R_2\ominus T_2=R_2\setminus T_2.
\end{cases}
$$
The distances between the bases $\Sigma_0,\Sigma_1,\Sigma_2,\Sigma_3$ are equal to $w_2-w_1,~w_1-w_0$ and $w_3-w_2$ respectively. 
Thus, we have:  
\begin{align*}\begin{cases}
\length(\Sigma_0)=\area(R_0)&=(w_1-w_0)\gamma(w_0)+(w_2-w_0)\gamma(w_1)+\\& \hspace{16ex}+(w_3-w_1)\gamma(w_2)+(w_3-w_2)\gamma(w_3)\\
\length(\Sigma_1)=\area(R_1)&=(w_1-w_0)\gamma(w_0)-w_0\gamma(w_1)+w_3\gamma(w_2)+(w_3-w_2)\gamma(w_3)\\
\length(\Sigma_2)=\area(R_2)&=-w_0\gamma(w_0)+w_3\gamma(w_2)+(w_3-w_2)\gamma(w_3)\\
\length(\Sigma_3)=\area(R_3)&=-w_0\gamma(w_0)+w_3\gamma(w_3).
\end{cases}\end{align*}
\begin{center}
\begin{tikzpicture}[scale=0.6]
\draw[thick, gray, ->] (3,-2)--(3,6);
\draw[thick, gray, ->] (-1,0)--(8,0);
\draw[thick] (0,0)--(7,0);
\draw[thick] (0,0)--(0,3)--(2,5)--(5,5)--(7,1)--(7,0);
\draw[ultra thick] (3,0)--(0,3);
\draw[ultra thick] (3,0)--(2,5);
\draw[ultra thick] (3,0)--(5,5);
\draw[ultra thick] (3,0)--(7,1);
\filldraw[color=black, fill=black, thick](0,3) circle (0.1);
\filldraw[color=black, fill=black, thick](2,5) circle (0.1);
\filldraw[color=black, fill=black, thick](5,5) circle (0.1);
\filldraw[color=black, fill=black, thick](7,1) circle (0.1);
\filldraw[color=black, fill=red!70, thick](0,0) circle (0.15);
\filldraw[color=black, fill=red!70, thick](2,0) circle (0.15);
\filldraw[color=black, fill=red!70, thick](5,0) circle (0.15);
\filldraw[color=black, fill=red!70, thick](7,0) circle (0.15);
\filldraw[color=black, fill=black, thick](3,0) circle (0.15);
\filldraw[color=black, fill=white, thick](4,0) circle (0.15);
\node[right,gray,scale=0.8] at (3,6) {$e_3$};
\node[below,gray,scale=0.8] at (8,0) {$e_1$};
\node[right, scale=0.8] at (3,4) {$T_1$};
\node[right, scale=0.8] at (1,3) {$T_0$};
\node[right, scale=0.8] at (5,2) {$T_2$};
\draw[thick] (9.5,0)--(24,0);
\draw[thick] (13.25,1.5)--(24,1.5);
\draw[thick] (16.25,2.5)--(24,2.5);
\draw[thick] (20.75,3.5)--(24,3.5);
\draw[thick] (9.5,0)--(13.25,1.5)--(16.25,2.5)--(20.75,3.5);
\draw[thick] (24,3.5)--(24,0);
\filldraw[color=black, fill=magenta, thick](16.25,2.5) circle (0.25);
\filldraw[color=black, fill=blue, thick](13.25,1.5) circle (0.25);
\filldraw[color=black, fill=violet, thick](9.5,0) circle (0.25);
\filldraw[color=black, fill=teal, thick](20.75,3.5) circle (0.25);
\filldraw[color=black, fill=cyan, thick](24,0) circle (0.25);
\filldraw[color=black, fill=orange, thick](24,3.5) circle (0.25);
\node[below] at (15,0) {$\Sigma_0$};
\node[below] at (17,1.5) {$\Sigma_1$};
\node[below] at (19,2.5) {$\Sigma_2$};
\node[above] at (22,3.5) {$\Sigma_3$};
\node[left] at (11.5,1) {$E_1$};
\node[below] at (14.1,2.8) {$E_0$};
\node[above] at (18.3,2.9) {$E_2$};
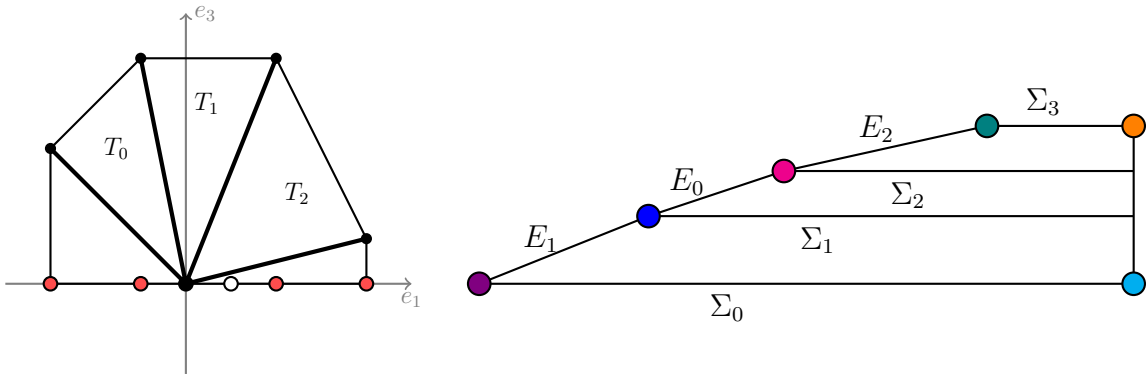
\node[below] at (12,-2) {Figure 8. The triangles $T_i$ and the fiber polygon $P.$};
\end{tikzpicture}
\end{center}
\end{exa}

\begin{proof}
First, suppose that the set $A$ is such that $0\in\conv(A).$ One can immediately find the following $4$ vertices of the sought polygon $P:$
\begin{enumerate}
\item the vertex $\big(w_k\gamma(w_k)-w_0\gamma(w_0),w_k-w_0\big)$ associated with the path from $(w_k,0,\gamma(w_k))$ to $(w_0,0,\gamma(w_0))$ along the edges $[\big(w_{k},0,\gamma(w_{k})\big),(0,1,0)]$ and 
$[(0,1,0), \big(w_{0},0,\gamma(w_{0})\big)]$ of the polytope $\Delta.$
\item the vertex $(0,0),$ associated with the path along the edge of $\Delta$ connecting $(w_k,0,0)$ and $(w_0,0,0);$ 
\item the vertex $(0,w_k-w_0)$ associated with the path from $(w_k,0,0)$ to $(w_0,0,0)$ along the edges $[(w_k,0,0),(0,1,0)]$ and $[(0,1,0),(w_0,0,0)]$ of the polytope $\Delta;$
\item the vertex $(\area(N),0)$ arising from the path going from $(w_k,0,\gamma(w_k))$ to $(w_0,0,\gamma(w_0))$ along the edges $[\big(w_{j+1},0,\gamma(w_{j+1})\big),\big(w_j,0,\gamma(w_j)\big)]$ of $N.$
\end{enumerate}
So, we have found $3$ edges of $P$: the edge $\Sigma_0$ connecting the first vertex in the list above with the second, the edge $\Sigma_k$ connecting the third vertex with the fourth and the edge connecting the second vertex with the third. Since the edges of the fiber polytope $P$ are Minkowski integrals of the faces of $\Delta,$ we conclude that $P$ has $k$ more edges $E_0,\ldots,E_{k-1}$ for $k$ facets of the form $\conv(\{(0,1,0),(w_j,0,\gamma(w_j)),(w_{j+1},0,\gamma(w_{j+1}))\}),~0\leqslant j<k.$ The edge $E_j$ is thus a shifted copy of the interval $[(0,0),\big(w_{j}\gamma(w_{j+1})-w_{j+1}\gamma(w_j),w_{j+1}-w_j\big)].$ 

Therefore, the only question that remains is in which order these $k$ edges are placed in $P.$ We will now show that this order is encoded by the sequence $Z=(z_1,\ldots,z_k)$ (see Definition \ref{def_comb_data}.) The slopes the edges $E_j$ should be placed in $P$ in the order respecting convexity of $P.$ Namely, their slopes should decrease as we move in $P$ from bottom (the edge $\Sigma_0$) to top (the edge $\Sigma_k$). Therefore, if the edge $E_i$ should be placed somewhere below the edge $E_j,$ then we have the following inequality: 

\begin{equation*}
\dfrac{w_{i}\gamma(w_{i+1})-w_{i+1}\gamma(w_i)}{w_{i+1}-w_i}>\dfrac{w_{j}\gamma(w_{j+1})-w_{j+1}\gamma(w_j)}{w_{j+1}-w_j}. 
\end{equation*}
Multiplying both sides by $-1,$ we obtain: 
\begin{equation*}
\dfrac{w_{i+1}\gamma(w_i)-w_{i}\gamma(w_{i+1})}{w_{i+1}-w_i}<\dfrac{w_{j+1}\gamma(w_j)-w_{j}\gamma(w_{j+1})}{w_{j+1}-w_j}. 
\end{equation*}
Finally, we note that the expressions in both sides of the inequality can be rewritten in a more familiar way: 
\begin{equation*}
w_i\cdot\dfrac{\gamma(w_i)-\gamma(w_{i+1})}{w_{i+1}-w_i}+\gamma(w_i)<w_j\cdot\dfrac{\gamma(w_j)-\gamma(w_{j+1})}{w_{j+1}-w_j}+\gamma(w_j),  
\end{equation*}
or, in the notation of Section \ref{answer}, we have: 

\begin{equation*}
\varphi_{\gamma}(r_i)=w_i\cdot r_i+\gamma(w_i)<w_j\cdot r_j+\gamma(w_j)=\varphi_{\gamma}(r_j).  
\end{equation*}
By definition of the sequence $Z,$ the latter means exactly that $i$ occurs earlier than $j$ in $Z.$

Recall that the edges $E_j$ are shifted copies of intervals of the form $$[(0,0),\big(w_{j}\gamma(w_{j+1})-w_{j+1}\gamma(w_j),w_{j+1}-w_j\big)].$$ The last thing that we need to observe is the equality $$-\area(T_j)=w_{j}\gamma(w_{j+1})-w_{j+1}\gamma(w_j).$$ The statement of Proposition \ref{descr_P} for the set $A$ such that $0\in\conv(A)$ follows immediately. The endpoints of the edges $E_j$ placed in the right order in the polygone $P$ form a sequence of vertices, and thus encode a sequence of paths in $\Delta.$ The projections of these paths onto $N\subset\Delta$ cut out triangles $T_j$ one by one, as described in Proposition \ref{descr_P}.

Now, let $A\subset \Z$ be a set such that $0\notin\conv(A).$ Without loss of generality, we can assume that  
$A\subset \Z_{\geqslant 0},$ and thus, we have $0<w_0<w_k.$ This case might seem quite different from the one considered above. Indeed, all the paths encoding the points of the fiber polygon $P$ end up at the point $(0,1,0).$ Another difference is that there can be only one possible sequence $Z,$ which is $(0,1,2,\ldots,k-1).$ 

\begin{center}
	\begin{tikzpicture}[scale=0.3]
	\draw[thick, gray,->] (0,0)--(9,18);
	\draw[thick, gray, ->] (9,0)--(-7,4);
	\draw[thick, gray, ->] (1,0)--(1,12);
	\draw[thick] (8,16)--(4,17)--(-5,17)--(-3,12);
	\draw[thick,dashed] (-3,12)--(0,9)--(4,8);
	\draw[thick] (1,5)--(-3,12);
	\draw[thick] (1,5)--(4,8);
	\draw[thick] (1,5)--(-5,17);
	\draw[thick] (1,5)--(4,17);
	\draw[thick] (1,5)--(8,16);
	\draw[ultra thick, cyan] (8,16)--(4,8)--(1,5);
	\draw[ultra thick, orange] (8,16)--(1,5);
	\draw[ultra thick, teal] (4,17)--(1,5);
	\draw[ultra thick, violet] (4,17)--(-5,17)--(-3,12);
	\draw[ultra thick, violet, dashed] (-3,12)--(0,9)--(1,5);
	\filldraw[color=black, fill=red!70, thick](8,16) circle (0.2);
	\filldraw[color=black, fill=red!70, thick](7,14) circle (0.2);
	\filldraw[color=black, fill=red!70, thick](4,8) circle (0.2);
	\filldraw[color=black, fill=red!70, thick](5,10) circle (0.2);
	\filldraw[color=black, fill=black, thick](1,5) circle (0.2);
	\filldraw[color=black, fill=black, thick](1,2) circle (0.2);
	\filldraw[color=black, fill=black, thick](4,17) circle (0.15);
	\filldraw[color=black, fill=black, thick](-5,17) circle (0.15);
	\filldraw[color=black, fill=black, thick](-3,12) circle (0.15);
	\filldraw[color=black, fill=black, thick](0,9) circle (0.15);
	\node[right,scale=0.75] at (1.3,2.1) {$(0,0,0)$};
	\node[left,scale=0.75] at (1,5) {$(0,1,0)$};
	\node[right,scale=0.75] at (4,8) {$(w_0,0,0)$};
	\node[right,scale=0.75] at (5,10) {$(w_1,0,0)$};
	\node[right,scale=0.75] at (7,14) {$(w_2,0,0)$};
	\node[right,scale=0.75] at (8,16) {$(w_3,0,0)$};
	\node[above,scale=0.75] at (4,17) {$(w_3,0,\gamma(w_3))$};
	\node[above,scale=0.75] at (-5,17) {$(w_2,0,\gamma(w_2))$};
	\node[left,scale=0.75] at (-3,12) {$(w_1,0,\gamma(w_1))$};
	\node[right,gray,scale=0.75] at (9,18) {$e_1$};
	\node[above,gray,scale=0.75] at (-7,4) {$e_2$};
	\node[above,gray,scale=0.75] at (1,12) {$e_3$};
	\node[below] at (4,-1) {Figure 9. The polytope $\Delta$ for $0\notin\conv(A)$.};
	\end{tikzpicture}
\end{center}

Let us describe the paths in $\Delta$ corresponding to the vertices of the fiber polygon $P.$ The first $4$ of them are colored in Figure 9, and the projections of rest of them are shown in Figure 10 below.

\begin{itemize}
\item[--] the path along the edge $[(w_k,0,0),(0,1,0)]$ in $\Delta$ (the vertex $(0,w_k)$);
\item[--] the path from $(w_k,0,0)$ to $(0,1,0)$ along the edges $[(w_k,0,0),(w_0,0,0)]$ and $[(w_0,0,0),(0,1,0)]$ (the vertex $(0,w_0)$);
\item[--] the path from $(w_k,0,\gamma(w_k))$ to $(0,1,0)$ along the edge $[(w_k,0,\gamma(w_k)),(0,1,0)];$ 
\item[--] the path from $(w_k,0,\gamma(w_k))$ to $(0,1,0)$ following the consecutive edges of the form  $[\big(w_{j+1},0,\gamma(w_{j+1})\big),\big(w_j,0,\gamma(w_j)\big)]\subset N\subset\Delta$ until $(w_0,0,\gamma(w_0))$ and then going along the edge $[(w_0,0,\gamma(w_0)),(0,1,0)]$ (vertex $(\area R_0,w_0)$);
\item[--] paths $(w_k,0,\gamma(w_k))$ to $(0,1,0),$ following the consecutive edges of the form  $[\big(w_{j+1},0,\gamma(w_{j+1})\big),\big(w_j,0,\gamma(w_j)\big)]$ until $(w_i,0,\gamma(w_i)),~i>0.$
\end{itemize}
Figure 10 shows the sets $R_j$ and the projections of the paths in $\Delta$ associated with vertices that have a non-zero first coordinate. 
\begin{center}
\begin{tikzpicture}[scale=0.6]
\filldraw[color=white, fill=violet!20, thick](0,0)--(8,0)--(8,1)--(7,4)--(4,3)--(3,1)--(0,0);
\draw[thick] (3,0)--(3,1)--(4,3)--(7,4)--(8,1)--(8,0)--(3,0);
\draw[ultra thick, violet](8,1)--(7,4)--(4,3)--(3,1)--(0,0);
\draw[thick, gray, ->] (-1,0)--(9,0);
\draw[thick, gray, ->] (0,-1)--(0,4);
\filldraw[color=black, fill=red!70, thick](3,0) circle (0.15);
\filldraw[color=black, fill=red!70, thick](4,0) circle (0.15);
\filldraw[color=black, fill=red!70, thick](7,0) circle (0.15);
\filldraw[color=black, fill=red!70, thick](8,0) circle (0.15);
\filldraw[color=black, fill=black, thick](3,1) circle (0.1);
\filldraw[color=black, fill=black, thick](4,3) circle (0.1);
\filldraw[color=black, fill=black, thick](7,4) circle (0.1);
\filldraw[color=black, fill=black, thick](8,1) circle (0.1);
\filldraw[color=black, fill=black, thick](0,0) circle (0.1);
\node[below right, scale=0.8] at (0,0) {$0$};
\node[below, scale=0.8] at (3,0) {$w_0$};
\node[below, scale=0.8] at (4,0) {$w_1$};
\node[below, scale=0.8] at (7,0) {$w_2$};
\node[below, scale=0.8] at (8,0) {$w_3$};
\node[below,scale=0.8] at (5,-0.5) {$R_0$};
\node[above,gray,scale=0.8] at (9,0) {$e_1$};
\node[left,gray, scale=0.8] at (0,4) {$e_3$};
\filldraw[color=white, fill=blue!20, thick](11,0)--(19,0)--(19,1)--(18,4)--(15,3)--(11,0);
\filldraw[color=white, fill=blue!30, thick](11,0)--(14,1)--(15,3);
\draw[thick] (14,0)--(14,1)--(15,3)--(18,4)--(19,1)--(19,0)--(14,0);
\draw[thick](19,1)--(18,4)--(15,3)--(14,1)--(11,0);
\draw[ultra thick,blue](19,1)--(18,4)--(15,3)--(11,0);
\draw[thick, gray, ->] (10,0)--(20,0);
\draw[thick, gray, ->] (11,-1)--(11,4);
\filldraw[color=black, fill=red!70, thick](14,0) circle (0.15);
\filldraw[color=black, fill=red!70, thick](15,0) circle (0.15);
\filldraw[color=black, fill=red!70, thick](18,0) circle (0.15);
\filldraw[color=black, fill=red!70, thick](19,0) circle (0.15);
\filldraw[color=black, fill=black, thick](14,1) circle (0.1);
\filldraw[color=black, fill=black, thick](15,3) circle (0.1);
\filldraw[color=black, fill=black, thick](18,4) circle (0.1);
\filldraw[color=black, fill=black, thick](19,1) circle (0.1);
\filldraw[color=black, fill=black, thick](11,0) circle (0.1);
\node[below right, scale=0.8] at (11,0) {$0$};
\node[below, scale=0.8] at (14,0) {$w_0$};
\node[below, scale=0.8] at (15,0) {$w_1$};
\node[below, scale=0.8] at (18,0) {$w_2$};
\node[below, scale=0.8] at (19,0) {$w_3$};
\node[below,scale=0.8] at (16,-0.5) {$R_1$};
\node[above right, scale=0.8] at (13,0.75) {$T_0$};
\node[above,gray,scale=0.8] at (20,0) {$e_1$};
\node[left,gray, scale=0.8] at (11,4) {$e_3$};
\end{tikzpicture}
\begin{tikzpicture}[scale=0.6]
\filldraw[color=white, fill=magenta!30, thick](0,0)--(8,0)--(8,1)--(7,4)--(0,0);
\pattern[pattern=north east lines, pattern color=magenta] (0,0)--(7,4)--(4,3)--(0,0);
\draw[thick] (3,0)--(3,1)--(4,3)--(7,4)--(8,1)--(8,0)--(3,0);
\draw[thick](8,1)--(7,4)--(4,3)--(3,1)--(0,0);
\draw[thick](8,1)--(7,4)--(4,3)--(0,0);
\draw[thick, gray, ->] (-1,0)--(9,0);
\draw[thick, gray, ->] (0,-1)--(0,4);
\draw[ultra thick, magenta](8,1)--(7,4)--(0,0);
\filldraw[color=black, fill=red!70, thick](3,0) circle (0.15);
\filldraw[color=black, fill=red!70, thick](4,0) circle (0.15);
\filldraw[color=black, fill=red!70, thick](7,0) circle (0.15);
\filldraw[color=black, fill=red!70, thick](8,0) circle (0.15);
\filldraw[color=black, fill=black, thick](3,1) circle (0.1);
\filldraw[color=black, fill=black, thick](4,3) circle (0.1);
\filldraw[color=black, fill=black, thick](7,4) circle (0.1);
\filldraw[color=black, fill=black, thick](8,1) circle (0.1);
\filldraw[color=black, fill=black, thick](0,0) circle (0.1);
\node[left, scale=0.8] at (3.5,2.8) {$T_1$};
\node[below right, scale=0.8] at (0,0) {$0$};
\node[below, scale=0.8] at (3,0) {$w_0$};
\node[below, scale=0.8] at (4,0) {$w_1$};
\node[below, scale=0.8] at (7,0) {$w_2$};
\node[below, scale=0.8] at (8,0) {$w_3$};
\node[above,gray,scale=0.8] at (9,0) {$e_1$};
\node[left,gray, scale=0.8] at (0,4) {$e_3$};
\node[below,scale=0.8] at (5,-0.5) {$R_2$};
\pattern[pattern=north east lines, pattern color=teal] (11,0)--(19,1)--(18,4)--(11,0);
\filldraw[color=white, fill=teal!20, thick](11,0)--(19,0)--(19,1)--(11,0);
\draw[thick] (14,0)--(14,1)--(15,3)--(18,4)--(19,1)--(19,0)--(14,0);
\draw[thick](19,1)--(18,4)--(15,3)--(14,1)--(11,0);
\draw[thick](19,1)--(18,4)--(15,3)--(11,0);
\draw[thick](19,1)--(18,4)--(11,0);
\draw[thick, gray, ->] (10,0)--(20,0);
\draw[thick, gray, ->] (11,-1)--(11,4);
\draw[ultra thick, teal](19,1)--(11,0);
\filldraw[color=black, fill=red!70, thick](14,0) circle (0.15);
\filldraw[color=black, fill=red!70, thick](15,0) circle (0.15);
\filldraw[color=black, fill=red!70, thick](18,0) circle (0.15);
\filldraw[color=black, fill=red!70, thick](19,0) circle (0.15);
\filldraw[color=black, fill=black, thick](14,1) circle (0.1);
\filldraw[color=black, fill=black, thick](15,3) circle (0.1);
\filldraw[color=black, fill=black, thick](18,4) circle (0.1);
\filldraw[color=black, fill=black, thick](19,1) circle (0.1);
\filldraw[color=black, fill=black, thick](11,0) circle (0.1);
\node[above, scale=0.8] at (19.3,1.5) {$T_2$};
\node[below right, scale=0.8] at (11,0) {$0$};
\node[below, scale=0.8] at (14,0) {$w_0$};
\node[below, scale=0.8] at (15,0) {$w_1$};
\node[below, scale=0.8] at (18,0) {$w_2$};
\node[below, scale=0.8] at (19,0) {$w_3$};
\node[below,scale=0.8] at (16,-0.5) {$R_3$};
\node[above,gray,scale=0.8] at (20,0) {$e_1$};
\node[left,gray, scale=0.8] at (11,4) {$e_3$};
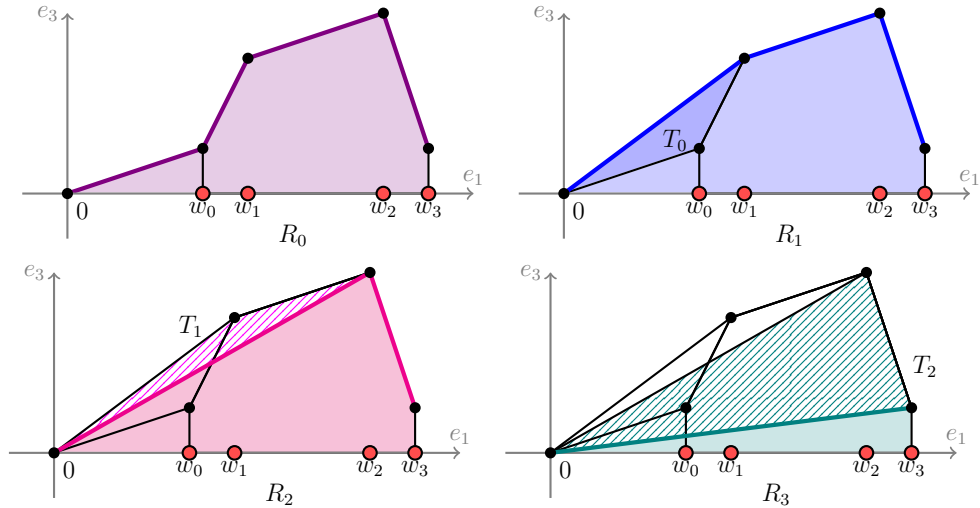
\node[below] at (10,-2) {Figure 10. The sets $R_j$ and the projections of the corresponding paths in $\Delta.$};
\end{tikzpicture}
\end{center} 
As shown in Figure 10, the triangle $T_j$ might lie outside the set $R_{j-1}$, and $R_j$ is then $R_{j-1}\cup T_j$. However, this case is covered in the construction, since we use the symmetric difference of sets to define $R_j.$ 
\end{proof}

\begin{lemma}\label{VolP}
In the same notation as above, we have the following equality:
\begin{multline}
\vol(\int_{\pi}\Delta)=\sum\limits_{j=1}^{k} S_{z_j}\Big(d_{z_j}+\sum\limits_{l=1}^{j-1}2d_{z_l}\Big)+\\+
(|w_0|-w_0)(w_k-w_0)\gamma(w_0)+(w_k+|w_k|)(w_k-w_0)\gamma(w_k),
\end{multline}
where $d_i=w_{i+1}-w_i$ and $S_i=w_{i+1}\gamma(w_i)-w_i\gamma(w_{i+1}).$
\end{lemma}

\begin{proof}
By Proposition \ref{descr_P}, the fiber polytope $P=\int_{\pi}\Delta$ is a union of $k$ rectangular trapezoids sitting on top of each other. The lengths of the bases as well as the heights of the trapezoids are already computed. Using the well-known formula for the area of a trapezoid and computing the sum of areas of all the $k$ trapezoids, we obtain the desired formula.
\end{proof}

The following formula is a straightforward corollary of Lemma \ref{VolP}. 

\begin{sledst}
In the same notation as above, for a finite set $A\subset \Z_{>0}$ and a covector $\gamma\in(\R^{|A|})^*$ with the corresponding subdivision $W\subset A,$ we have the following formula:
\begin{multline}\label{volP_poly}
\vol(\int_{\pi}\Delta)=w_1(w_1-w_0)\gamma(w_0)+(w_k-w_{k-1})(2w_k+w_{k-1}-2w_0)\gamma(w_k)+\\+\sum\limits_{j=1}^{j=k-1}(w_{j+1}-w_{j-1})(w_{j-1}+w_j+w_{j+1}-2w_0)\gamma(w_j).
\end{multline}
\end{sledst}

\begin{exa}
	For $A=[1,n]\cap\Z$ and a concave function $\gamma\colon A\to\Z_{\geqslant 0}$, we have 
	$$\vol(\int_{\pi}\Delta)=2\gamma(1)+(3n-3)\gamma(n)+\sum_{m=2}^{n-1} 2(3m-2)\gamma(m).$$
\end{exa}

\subsection{Step 5: the singularities at infinity}\label{sub_FPS}
In this subsection we compute the other two summands in the right-hand side of (\ref{rel3}). 
Without loss of generality, we can assume that the polytope $\Delta$ satisfies the following condition.

\begin{predpol}\label{primitive}
	For every facet of $\Delta$ the image of its primitive normal covector under the projection forgetting the first coordinate is also primitive.
\end{predpol}

\begin{rem}\label{prim_expl1}
	One can always achieve the condition from Assumption \ref{primitive} using the change of variables $\check{x}=x,~\check{y}^{M!}=y,~\check{t}^{M!}=t$ for $M$ big enough. Moreover, under this change of variables, both sides of the equation (\ref{rel3}) are multiplied by $(M!)^2.$ The latter means that the formulas for the contributions of singular points at infinity, which we are going to obtain in this subsection, will work independently of Assumption \ref{primitive}.
\end{rem}

First we introduce a bit of notation. Let $\tilde{A}_{\gamma}\subset\Z^3$ be the support of the polynomial $f_t(x)=\sum\limits_{p\in A}(q_p+v_pt^{\gamma(p)})x^p.$ By $\Gamma_j,~0\leqslant j\leqslant k-1,$ we denote the facets of the form $\conv(\{(0,1,0),(w_j,0,\gamma(w_j)),(w_{j+1},0,\gamma(w_{j+1}))\})\subset\Delta.$ 

Each facet $\Gamma_j$ is contained in the hyperplane $H_j$ given by an equation of the form $h_j(e_1,e_2,e_3)=\lambda_j$. The statement below can be easily shown by a straightforward computation. 
\begin{utver}
	In the same notation as above, the facet $\Gamma_j$ lies in a hyperplane $H_j$ defined by the following equation:
	\begin{multline}\label{h_j}
	\big(\gamma(w_j)-\gamma(w_{j+1})\big)e_1+\big(w_{j+1}\gamma(w_j)-w_j\gamma(w_{j+1})\big)e_2+\big(w_{j+1}-w_j\big)e_3=\\=\big(w_{j+1}\gamma(w_j)-w_j\gamma(w_{j+1})\big).
	\end{multline}
\end{utver}

\begin{rem}
	In the notation of Section \ref{answer}, one can rewrite the equality (\ref{h_j}) as follows: 
	$$(w_{j+1}-w_j)r_j\cdot e_1+(w_{j+1}-w_j)\varphi_{\gamma}(r_j)\cdot e_2+(w_{j+1}-w_j)e_3=(w_{j+1}-w_j)\varphi_{\gamma}(r_j).$$
\end{rem}

We now construct a sequence of sets $B^j$ and the corresponding sequence of integers  $i^{j}=(i_1^{j},i_2^{j},\ldots)$ as follows. Set $B_1^{j}=\tilde{A}_{\gamma}\cap\Gamma_j.$ For every $l>1,$ we define  $$B_l^{j}=B_{l-1}^{j}\cup(\tilde{A}_{\gamma}\cap\{h_j(e_1,e_2,e_3)=\lambda_j-(l-1)\}),$$ depending on the way $\Delta$ is positioned relative to the hyperplane containing $\Gamma_j.$ 
Finally, for every $l\geqslant 1,$ we set $$i_l^{j}=\ind_{v}(B_l^{j}).$$

It is clear that for every $r,$ the element $i_r^{j}$ divides $i_{r-1}^{j}.$ Moreover, since the set $A$ affinely generates the lattice $\Z,$ any such sequence stabilizes to $1$.

We are interested in the singularities of the closure of the curve $D$ in the toric surface $X_P,$ where $P=\int_{\pi}\Delta.$ The following result describes the singularities on the orbits of $X_P,$ corresponding to the Minkowski integrals of facets $\Gamma_j.$ 

\begin{utver}
	In the same notation as above, the singularities of $\overline{D}\subset X_P$ on the orbits of $X_P$ corresponding to the Minkowski integrals of facets $\Gamma_j$ are $\dfrac{\vol{\Gamma_j}}{i_1^j}$ copies of the  $i^{j}$--forking paths singularity. 
\end{utver}
\begin{proof}
Using a suitable change of variables (which exists due to Assumption \ref{primitive}), one can reduce the computation for the given facet $\Gamma_j$ exactly to the case considered in Example \ref{fpsgen}. The desired statement follows immediately. 
\end{proof}

We are looking for the sum $\sum_{s\in \mathrm{FPS}}\chi([\text{Milnor fiber of~} s]\cap\CC^2).$ Denote by $C^{j}_{\gamma}$ the contribution of the facets $\Gamma_j$ to this sum. The following result follows directly from the formula for the Euler characteristic of the Milnor fiber of a forking--path singularity given in Proposition \ref{chifps}. 
\begin{utver}
	The following equality holds: 
	\begin{equation} \label{fpsc_def}
	-C^{j}_{\gamma}=\vol(\Gamma_j)\sum\limits_1^{\infty}(i_l^{j}-1).
	\end{equation}
\end{utver}

Let us compute $C^j_{\gamma}$ in terms of the support set $\tilde{A}.$ Suppose that $(m,0,\gamma(m))$ is the first point ``encountered'' while shifting the plane $L_j.$ Then a straightforward computation implies that the distance from this point to the plane $L_j$ is equal to 

\begin{equation}\label{dist}
\ell_m=\dfrac{1}{\vol(\Gamma_j)}\Bigg((w_{j+1}-m)\gamma(w_j)+(m-w_{j})\gamma(w_{j+1})-(w_{j+1}-w_j)\gamma(m)\Bigg).
\end{equation}

Therefore, the first $\ell_m$ elements of the sequence $i^{j}$ are equal to $\gcd(w_{j+1},w_j),$ and the element number $\ell_{m}+1$ is equal to $\gcd(w_j,m,w_{j+1}).$ 

Suppose that we know the order in which the points $(m,0,\gamma(m)),~m\neq w_j,w_{j+1},$ occur in the sequence $B_l^{j}.$ It can be encoded by a sequence $M^j=(m^j_l),~1\leqslant l\leqslant |A|-2$ of the integer numbers in $A\setminus\{w_j,w_{j+1}\}.$

Recall that in Section \ref{answer} the sequence $M^j$ was defined in a different way, as the sequence of elements in $A\setminus\{w_j,w_{j+1}\}$ placed in decreasing order with respect to the values of the monomials $\gamma(p)\odot X^{\odot p}$ attained at the root $r_j$ of the tropical polynomial $\varphi_{\gamma}(X).$ 
\begin{utver}
	The two ways of defining the sequence $M^j$ are equivalent. 
\end{utver}

\begin{proof}
	The point $p\in A$ occurs in the sequence $M^j$ earlier than $m\in A$ if the distance $\ell_p$ between the plane $L_j$ and the point $(p,0,\gamma(p))$ is smaller than the distance $\ell_m$ between $L_j$ and $(m,0,\gamma(m)).$ Using (\ref{dist}), the latter can be rewritten as follows: 
	\begin{multline*}
	(w_{j+1}-m)\gamma(w_j)+(m-w_{j})\gamma(w_{j+1})-(w_{j+1}-w_j)\gamma(m)>\\>(w_{j+1}-p)\gamma(w_j)+(p-w_{j})\gamma(w_{j+1})-(w_{j+1}-w_j)\gamma(p),
	\end{multline*}
	or, equivalently, 
	\begin{equation}
	m\dfrac{\gamma(w_j)-\gamma(w_{j+1})}{w_{j+1}-w_j}+\gamma(m)<p\dfrac{\gamma(w_j)-\gamma(w_{j+1})}{w_{j+1}-w_j}+\gamma(p),
	\end{equation}
	which concludes the proof, since $r_j=\dfrac{\gamma(w_j)-\gamma(w_{j+1})}{w_{j+1}-w_j}.$
\end{proof}

Recall that by $b^j$ we denoted the finite sequence of numbers defined recursively as follows: $b^j_0=\gcd(w_j,w_{j+1})$ and $b^j_l=\gcd(b^j_{l-1},m_l)$ for $l>0.$ Note that since the sequence $i^{j}$ stabilizes to $1,$ then so does the sequence $b^j.$ 

\begin{utver}\label{fps_C}
	In the same notation as above, the number $C^{j}_{\gamma}$ can be computed via the following formula: 
	\begin{equation}\label{fps_contrib2}
	C^j_{\gamma}=\sum\limits_{l\geqslant 1}\Big((w_{j+1}-w_j)\gamma(m_l^j))+(m_l^j-w_{j+1})\gamma(w_j)+(w_j-m_l^j)\gamma(w_{j+1})\Big)(b_{l-1}^j-b_l^j).
	\end{equation}
\end{utver}

\begin{proof}
	To prove this statement, we can write down the expression for $C^j_{\gamma}$ using (\ref{fpsc_def}), expand it and eliminate the terms with opposite signs. 
\end{proof}

The following statement is a straightforward corollary of Proposition \ref{fps_C}. 

\begin{sledst}
	In the same notation as above, 
	\begin{itemize}
		\item [i.] If $\gcd(w_j,w_{j+1})=1,$ then $C^j_{\gamma}=0;$
		\item [ii.] The terms enumerated by the indices $l$ such that $b^j_{l-1}=b^j_{l}$ do not contribute to the sum (\ref{fps_contrib2});
		\item [ii.] Let $m_R$ be the first term of the sequence $M^j$ such that $b^j_R=1,$ then for any $l>R,$ the terms enumerated by $l$ do not contribute to the sum (\ref{fps_contrib2}).
	\end{itemize}
\end{sledst}

\begin{exa}
	If $M^j$ is such that $b^j$ stabilizes at $1$ after the first element $m_1^j=m,$ then we have $$C^j_{\gamma}=\Big((m-w_{j+1})\gamma(w_j)+(w_j-m)\gamma(w_{j+1})+(w_{j+1}-w_j)\gamma(m)\Big)(\gcd(w_j,w_{j+1})-1).$$
\end{exa}

The polytope $\Delta$ has $4$ other facets, whose contribution is yet to be studied: 
\begin{itemize}
	\item[--] The facet $\conv(\{(0,1,0),(w_0,0,0),(w_k,0,0)\});$
	\item[--] The base $N\subset\Delta$ of the pyramid $\Delta;$
	\item[--] The facet $Q_1=\conv(\{(0,1,0),(w_0,0,0),(w_0,0,\gamma(w_0))\});$
	\item[--] The facet $Q_2=\conv(\{(0,1,0),(w_k,0,0),(w_k,0,\gamma(w_k))\}).$
\end{itemize}

The first of the facets does not contribute to the Euler characteristic of the curve $D.$ Indeed, on one hand, the closure of the curve $\mathcal C$ in the toric variety $X_{\Delta}$ intersects the corresponding orbit transversally due to genericity of the coefficients $q_p,$ so, the contribution of this facet can be computed using the same method as for the facets $\Gamma_j$ (see Section \ref{proj_sing} for details). At the same time, the intersection with $\tilde{A}_{\gamma}$ contains all the points $(p,0,0),~p\in A,$ and since $A$ affinely generates the lattice $\Z,$ which automatically makes the sought contribution equal to $0.$ 

One can easily show that the truncated system $\{f(x,y,t)^N=x\dfrac{\partial f(x,y,t)}{\partial x}=0\}$ has exactly  $\area(N)-\gamma(w_0)(w_1-w_0)-\gamma(w_k)(w_k-w_{k-1})$ roots in $\CC^2,$ and for generic coefficients $q_p$ and $v_p$ their second coordinates of those roots are all pairwise distinct and different from the roots for the univariate polynomials $q_{w_0}+v_{w_0}t^{\gamma(w_0)}$ and $q_{w_k}+v_{w_k}t^{\gamma(w_k)}.$ Thus, they do not contribute to the singularities of the curve $D$ at infinity. Therefore, the only contributions that we still need to deal with are the ones coming from the facets $Q_1$ and $Q_2.$

In fact, instead of computing these contributions, we will show that they only depend on the values  $\gamma(a_0),\gamma(a_{|A|-1})$ and the distances $a_1-a_0$ and $a_{|A|-1}-a_{|A|-2},$ but not on the subdivision $W\subset A.$ The latter means that taking these contributions into accout while computing the support function of the polytope $\mathcal M_A$ would only shift the polytope and not change it in any other way. 

Note that the closure of the complete intersection $\mathcal C$ does not intersect the $2$--dimensional orbit of the toric variety $X_{\Delta}$ corresponding to the facet $Q_1,$ but it does intersect the $1$--dimensional orbit corresponding to the horizontal edge $[(w_0,0,0),(w_0,0,\gamma(w_0))].$ To compute the contribution of the corresponding intersection points, we will use an auxiliary map, namely, the projection $\tilde{\pi}\colon\CC^3\to\CC^2$ forgetting the third coordinate. 
The monomial change of variables $x=\check x,~t=\check t,~y=\check{x}^{w_0} \check{y}$ makes the facet $Q_1$ vertical with respect to this projection, and reduces our case to the following one (see Figure 11 below). 
\begin{center}
\begin{tikzpicture}[scale=0.7]
\draw[thick, gray,->] (0,0)--(-4,-2);
\draw[thick, gray,->] (0,0)--(0,5);
\draw[thick, gray,->] (0,0)--(8,0);
\draw[ultra thick] (0,3)--(2,5)--(5,5)--(7,1)--(7,0);
\draw[ultra thick, dashed] (0,0)--(7,0);
\draw[ultra thick, dashed] (0,0)--(0,3);
\draw[ultra thick] (-2,-1)--(0,3);
\draw[ultra thick] (-2,-1)--(2,5);
\draw[ultra thick] (-2,-1)--(5,5);
\draw[ultra thick] (-2,-1)--(7,1);
\draw[ultra thick] (-2,-1)--(7,0);
\draw[ultra thick, dashed] (-2,-1)--(0,0);
\filldraw[color=black, fill=white, thick](4,0) circle (0.15);
\filldraw[color=black, fill=red!70, thick](0,0) circle (0.15);
\filldraw[color=black, fill=red!70, thick](2,0) circle (0.15);
\filldraw[color=black, fill=red!70, thick](5,0) circle (0.15);
\filldraw[color=black, fill=red!70, thick](7,0) circle (0.15);
\filldraw[color=black, fill=black, thick](0,3) circle (0.15);
\filldraw[color=black, fill=black, thick](2,5) circle (0.15);
\filldraw[color=black, fill=black, thick](5,5) circle (0.15);
\filldraw[color=black, fill=black, thick](7,1) circle (0.15);
\filldraw[color=black, fill=black, thick](-2,-1) circle (0.15);
\node[above,gray] at (8,0) {$e_1$};
\node[left,gray] at (0,5) {$e_3$};
\node[above,gray] at (-4,-2) {$e_2$};
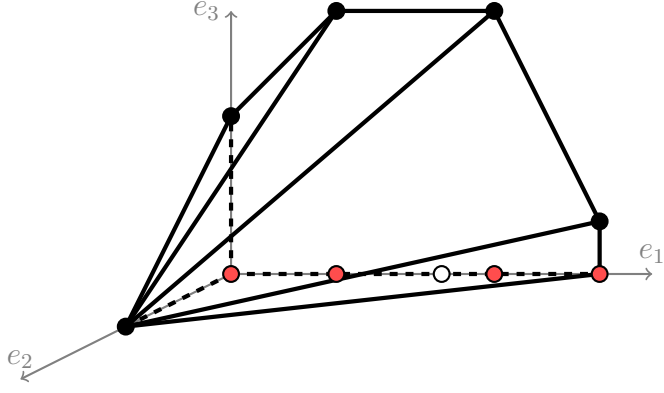
\node[below] at (4,-2.5) {Figure 11. The polytope $\Delta$ after the change of coordinates.};
\end{tikzpicture}
\end{center} 

Let $t_0$ be one of the roots for the polynomial $q_{a_0}+v_{a_0}t^{\gamma(a_0)}.$ The hypersurface $\mathcal H=\{\tilde{f}(\check x,\check y,\check t)=0\}$ intersects the orbit corresponding to the edge $[(0,0,0),(0,0,\gamma(a_0))]$ transversally at the point $(0,0,t_0).$ Therefore, near the point $(0,0,t_0),$ this hypersurface is a graph of a smooth function $\varphi(\check x,\check y),$ moreover, the restriction of the projection $\tilde{\pi}$ to $\mathcal H$ near $(0,0,t_0)$ is a diffeomorphism. Due to this obervation, it suffices to consider the projection $\tilde{\pi}\mathcal C$ near $(0,0)$ instead of the curve $\mathcal C$ near the point $(0,0,t_0)$ itself. The second polynomial of the system defining the curve $\mathcal C$ is $\tilde{g}(\check x,\check y,\check t)=\sum\limits_{p\in A} p(q_p+v_p\check t^{\gamma(p)})\check x^{p-a_0}.$ Let us substitute $\varphi(\check x,\check y)$ instead of $\check t.$ Then we obtain $\tilde{g}(\check x,\check y,\varphi(\check x,\check y))= \sum\limits_{p\in A} p(q_p+v_p\varphi(\check x,\check y)^{\gamma(p)})\check x^{p-a_0}.$ Finally, note that we have the equality $$\tilde{f}(\check x,\check y,\varphi(\check x, \check y))=\sum\limits_{p\in A} (q_p+v_p\varphi(\check x,\check y)^{\gamma(p)})\check x^{p-a_0}-\check y=0.$$

Thus, the projection $\tilde{\pi}(\mathcal C)$ near $(0,0)$ is given by the formula
$$\tilde{g}(\check x,\check y,\varphi(\check x,\check y))-a_0\cdot\tilde{f}(\check x,\check y,\varphi(\check x,\check y))=0,$$

or, equivalently, we have

$$a_0\cdot \check y+\sum\limits_{p\in A} (p-a_0)(q_p+v_p\varphi(\check x,\check y)^{\gamma(p)})\check x^{p-a_0}=0.$$

Finally, note that the Taylor series of $\varphi(\check x,\check y)$ near $(0,0,t_0)$ can be written as follows: $\varphi(\check x,\check y)=t_0+\alpha \check x+\beta \check y+ h.o.t.$ So, the only thing that we still need to do is to substitute this series into the defining equation computed above and find the vertices of the corresponding Newton diagram. Thus we obtain 
\begin{multline*}
a_0\cdot \check y+\sum\limits_{p\in A} (p-a_0)(q_p+v_p(t_0+\alpha \check x+\beta \check y+ h.o.t.)^{\gamma(p)}\check x^{p-a_0})=\\=a_0\cdot \check y+(a_1-a_0)(q_{a_1}+v_{a_1}t_0^{\gamma(a_1)})\check x^{a_1-a_0}+h.o.t.
\end{multline*}

Note that since $0\notin A,$ the coefficient at $y$ is non-zero. Moreover for a generic choice of numbers $q_p$ and $v_p$ the sets of roots for the polynomials $q_{a_0}+v_{a_0}t^{\gamma(a_0)}$ and 
$q_{a_1}+v_{a_1}t^{\gamma(a_1)}$ do not intersect, thus, the coefficient at $x^{a_1-a_0}$ is non-zero as well. Thus, we have shown that the contribution of the edge $Q_1$ does not depend on the subdivision $W,$ but only on $\gamma(a_0)=\length(Q_1)$ and the distance $a_1-a_0.$ Similarly, the contribution of the edge $Q_2$ depends on $\gamma(a_{|A|-1})$ and the distance $a_{|A|-1}-a_{|A|-2}.$ 

\subsection{Step 6: computing the support function of the Morse polytope}\label{sol_eq}

Combining all the results obtained in the subsections above, we obtain the following system of equations. 
\begin{equation}\label{e123}
\begin{cases}
\chi(\mathcal A_1)+2|2\mathcal A_1|+2|\mathcal A_2|=-\area(N)=w_0\gamma(w_0)-w_k\gamma(w_k)-\sum\limits_{j=0}^{k-1} S_j\\
|\mathcal A_2|=\area(N)-\gamma(w_0)-\gamma(w_k)=\sum\limits_{j=0}^{k-1} S_j+(w_k-1)\gamma(w_k)-(w_0+1)\gamma(w_0)\\
\chi(\mathcal A_1)-|\mathcal A_2|=-\sum\limits_{j=1}^{k} S_{z_j}\Big(d_{z_j}-\sum\limits_{l=1}^{j-1}2d_{z_l}\Big)-
(|w_0|-w_0)(w_k-w_0)\gamma(w_0)\\\hspace{82pt}-(w_k+|w_k|)(w_k-w_0)\gamma(w_k)-c_1\gamma(w_0)-c_2\gamma(w_k)-\sum\limits_{j=0}^{j=k-1}C^j_{\gamma},
\end{cases}
\end{equation} 
where $d_i=w_{i+1}-w_i,~S_i=w_{i+1}\gamma(w_i)-w_i\gamma(w_{i+1}),$ and $c_1,c_2$ are some constants not depending on $\gamma.$

The sought number $2|2\mathcal A_1|+|\mathcal A_2|$ can be easily extracted from the system (\ref{e123}).

Thus we obtain the following result. 

\begin{lemma}
	In the same notation as above, we have the following equality:	
	\begin{multline}\label{sol_eq123}
	2|2\mathcal A_1|+|\mathcal A_2|=\sum\limits_{j=1}^{k} S_{z_j}\Big(d_{z_j}-3+\sum\limits_{l=1}^{j-1}2d_{z_l}\Big)+\sum\limits_{j=0}^{j=k-1}C^j_{\gamma}\\+
	\big((|w_0|-w_0)(w_k-w_0)+3w_0+c_1+2\big)\gamma(w_0)\\+\big((w_k+|w_k|)(w_k-w_0)-3w_k+c_2+2\big)\gamma(w_k),
	\end{multline}
	where $d_i=w_{i+1}-w_i$ and $S_i=w_{i+1}\gamma(w_i)-w_i\gamma(w_{i+1}).$
\end{lemma}

\begin{rem}
Note that since $c_1$ and $c_2$ are some constants that do not depend on $\gamma,$ we can slightly simplify (\ref{sol_eq123}) by replacing $3w_0+c_1+2$ and $-3w_k+c_2+2$ with $c_1$ and $c_2$ respectively. 
\end{rem}

Finally, the value of the sought support function $\mu_A$ at a generic integer covector ${\gamma\in(\R^{|A|})^*}$ is exactly the sum $2|2\mathcal A_1|+|\mathcal A_2|,$ which concludes the proof of Theorem \ref{support}.  

\begin{rem}\label{maxwell_supp}
From the same system of equations (\ref{sol_eq123}), one can extract the number $|2\mathcal A_1|,$ which yields the support function of the Maxwell stratum. In general, to compute the support function of $h_c^ah_m^b,$ where $h_c$ and $h_m$ are defining polynomials of the caustic and the Maxwell stratum respectively, we need to find the sum $b|2\mathcal A_1|+a|\mathcal A_2|.$
\end{rem} 

{\bf Arina Voorhaar (Arkhipova)}  \href{mailto:arina.voorhaar@gmail.com}{\nolinkurl{arina.voorhaar@gmail.com}};\\

National Research University Higher School of Economics, Russian Federation,\\ Department of Mathematics, Usacheva ul. 6, Moscow, 119048; \href{mailto:aarhipova@hse.ru}{\nolinkurl{aarhipova@hse.ru}};\\ 

University of Geneva, Switzerland, Department of Mathematics, \\rue du Conseil-Général 7-9, Gen\`{e}ve, 1205; \href{mailto:Arina.Arkhipova@unige.ch}{\nolinkurl{Arina.Arkhipova@unige.ch}}.
\end{document}